\documentclass[a4paper]{amsart}
\usepackage{amsmath,amsthm,amssymb,latexsym,epic,bbm,comment}
\usepackage{graphicx,enumerate,stmaryrd}
\usepackage[all]{xy}
\xyoption{2cell}

\usepackage[active]{srcltx}
\usepackage[parfill]{parskip}

\newtheorem{theorem}{Theorem}
\newtheorem{lemma}[theorem]{Lemma}

\newtheorem{corollary}[theorem]{Corollary}
\newtheorem{proposition}[theorem]{Proposition}
\newtheorem{example}[theorem]{Example}
\newtheorem{problem}[theorem]{Problem}

\newcommand{\tto}{\twoheadrightarrow}

\font\sc=rsfs10
\newcommand{\cQ}{\sc\mbox{Q}\hspace{1.0pt}}
\newcommand{\cC}{\sc\mbox{C}\hspace{1.0pt}}

\newcommand{\cP}{\sc\mbox{P}\hspace{1.0pt}}

\font\scc=rsfs7
\newcommand{\ccC}{\scc\mbox{C}\hspace{1.0pt}}

\newcommand{\ccP}{\scc\mbox{P}\hspace{1.0pt}}
\newcommand{\ccQ}{\scc\mbox{Q}\hspace{1.0pt}}

\begin{document}

\title[$2$-categories associated with dual projection functors]{Finitary $2$-categories 
associated with\\ dual projection functors}
\author{Anna-Louise Grensing and Volodymyr Mazorchuk}

\begin{abstract}
We study finitary $2$-categories associated to dual projection functors for
finite dimensional associative algebras. In the case of path algebras of admissible tree quivers 
(which includes all Dynkin quivers of type $A$)
we show that the monoid generated by dual projection functors is the Hecke-Kiselman monoid of 
the underlying quiver and also obtain a presentation for the monoid of indecomposable subbimodules 
of the identity bimodule. 
\end{abstract}

\maketitle

\section{Introduction}\label{s1}

Study of $2$-categories of additive functors operating on a module category of a finite dimensional
associative algebra is motivated by recent advances and applications of categorification philosophy,
see \cite{CR,Ro,KhLa,Ma} and references therein. Such $2$-categories appear as natural $2$-analogues 
of finite dimensional algebras axiomatized via the notion of {\em finitary} $2$-categories 
as introduced in \cite{MM1}. The series \cite{MM1,MM2,MM3,MM5,MM6} of papers develops basics of the structure 
theory and the $2$-representation theory for the so-called {\em fiat $2$-categories}, that is finitary 
$2$-categories having a weak involution and adjunction morphisms. Natural examples of such fiat 
$2$-categories are $2$-categories generated by {\em projective} functors, that is functors given by
tensoring with projective bimodules, see \cite[Subsection~7.3]{MM1}. Fiat $2$-categories also naturally 
appear as quotients of $2$-Kac-Moody algebras from \cite{KhLa,Ro,We}, see \cite[Subsection~7.1]{MM2} 
and \cite[Subsection~7.2]{MM5} for detailed explanations. There are also many natural constructions which 
produce new fiat $2$-categories from known ones, see e.g. \cite[Section~6]{MM6}.

Despite of some progress made in understanding fiat $2$-categories in the papers mentioned above, 
the general case of finitary $2$-categories remains very mysterious with the only general result 
being the abstract $2$-analogue of the Morita theory developed in \cite{MM4}. One of the major difficulties 
is that so far there are not that many natural 
examples of finitary $2$-categories which would be ``easy enough'' for any kind of sensible understanding.
In \cite{GrMa}, inspired by the study of the so-called {\em projection} functors in \cite{Gr,Pa},
we defined a finitary $2$-category which is a natural $2$-analogue of the semigroup algebra of the
so-called {\em Catalan monoid} of all order-decreasing and order-preserving transformations of a finite chain.
This $2$-category is associated to the path algebra of a type $A$ Dynkin quiver with a fixed uniform 
orientation (meaning that all edges are oriented in the same direction).

The main aim of the present paper is to make the next step and consider a similarly defined $2$-category
for an arbitrary orientation of a type $A$ Dynkin quiver and, more generally, for any admissible orientation
of an arbitrary tree quiver. There is one important difference, which we will  now explain, between 
this general case and the case of a uniform orientation in type $A$. Basic structural properties
of a finitary $2$-category are encoded in the so-called {\em multisemigroup} of this $2$-category as defined
in \cite[Subsection~3.3]{MM2}. Elements of this multisemigroup are isomorphism classes of {\em indecomposable}
$1$-morphisms in our $2$-category. It turns out that for a uniform orientation of a type $A$ Dynkin quiver
any composition of projection functors is either indecomposable or zero. This fails in all other cases
in which the orientation is not uniform as well as for all admissible tree quivers outside type $A$. 
This is the principal added difficulty of the present paper compared to \cite{GrMa}.

For technical reasons it turns out that it is more convenient to work with a dual version of projection
functors, which we simply call {\em dual projection functors}. Roughly speaking these are the right exact 
functors given by maximal subfunctors of the identity functor. The first part of the paper is devoted to 
some basic structure theory for such functors. This is developed in Section~\ref{s3} after various preliminaries
collected in Section~\ref{s2}. In particular, in Proposition~\ref{prop8} we make the connection between
projection and dual projection functors very explicit. This, in particular, allows us to transfer, for free, 
many results of \cite{Gr,Pa} to our situation. 

Section~\ref{s4} contains basic preliminaries on $2$-categories. In Section~\ref{s5} we define finitary
$2$-categories given by dual projection functors and also finitary $2$-categories given by non-exact ancestors 
of  dual projection functors which we call {\em idealization functors}.  Section~\ref{s7} is the main part of
the paper and contains several results. This includes a classification of indecomposable dual projection functors
in Theorem~\ref{thm31} and also the statement that composition of indecomposable dual projection 
functors for any admissible orientation of a tree quiver is indecomposable, see Proposition~\ref{prop32}.
Our classification is based on a generalization of the Dyck path combinatorics in application to subbimodules
of the identity bimodule for admissible tree quivers as described in Subsections~\ref{s7.3}, \ref{s7.4}, 
\ref{s7.5} and \ref{s7.7}.

Proposition~\ref{prop32} mentioned above implies that the multisemigroup of the $2$-category of dual projection 
functors associated to any admissible orientation of a tree quiver is, in fact, an ordinary semigroup. 
This observation automatically makes this semigroup an interesting object of study.
In Section~\ref{s8} we give a presentation for this semigroup in Theorem~\ref{thm55}
and also for the semigroup of all idealization functors in Theorem~\ref{thm54}. Our proof of 
Theorem~\ref{thm54} is rather elegant, it exploits the idea of decategorification: the canonical
action of our $2$-category on the underlying module category gives rise to a linear representation of 
a certain Hecke-Kiselman monoid from \cite{GM2}. Proof of Theorem~\ref{thm54} basically reduces to verification 
that this representation is effective (in the sense that different elements of the monoid are represented
by different linear transformations). This effectiveness was conjectured in \cite{GM2} and proved in \cite{Fo}.
Theorem~\ref{thm55} requires more technical work as the monoid of indecomposable dual projection functors is
not a Hecke-Kiselman monoid on the nose, but after some preparation it also reduces to a similar argument.
\vspace{5mm}

\noindent
{\bf Acknowledgment.} 
The first author is supported by the priority program SPP 1388 of the German Science Foundation.
The second author is partially supported by the Swedish Research Council,
Knut and Alice Wallenbergs Stiftelse and the Royal Swedish Academy of Sciences.
We thank the referee for a very careful reading of the original manuscript and many
useful comments and suggestions which corrected several inaccuracies and substantially improved exposition. 
\vspace{20mm}

\section{Preliminaries}\label{s2}

\subsection{Notation and setup}\label{s2.1}

In this paper we work over a fixed field $\Bbbk$ which for simplicity is assumed to be algebraically closed.
All categories and functors considered in this paper are supposed to be $\Bbbk$-linear, that is enriched over
$\Bbbk\text{-}\mathrm{Mod}$. If not explicitly stated otherwise, by a module we always mean a {\em left} module.

For a finite dimensional associative $\Bbbk$-algebra
$A$ we denote by $A\text{-}\mathrm{mod}$ the (abelian) category of all finitely generated $A$-modules.
By $A\text{-}\mathrm{Mod}$  we denote the (abelian) category of all $A$-modules.
We also denote by $A\text{-}\mathrm{proj}$ the (additive) category of all finitely generated projective $A$-modules
and by $A\text{-}\mathrm{inj}$ the (additive) category of all finitely generated injective $A$-modules.

We denote by $\mathrm{mod}\text{-}A$ the category of all finitely generated right $A$-modules and define
$\mathrm{proj}\text{-}A$ and $\mathrm{Mod}\text{-}A$ respectively.

We denote by $A\text{-}\mathrm{mod}\text{-}A$ the category of all finitely generated $A\text{-}A$--bimodules.
Denote by $\mathcal{AF}_{A}$ the category of all additive $\Bbbk$-linear endofunctors of 
$A\text{-}\mathrm{mod}$. This is an abelian category since $A\text{-}\mathrm{mod}$ is abelian.

Abusing notation, we write $*$ for both the $\Bbbk$-duality functors 
\begin{displaymath}
\mathrm{Hom}_{\text{-}\Bbbk}({}_-,\Bbbk):A\text{-}\mathrm{mod}\to  \mathrm{mod}\text{-}A\quad\text{ and }\quad
\mathrm{Hom}_{\Bbbk\text{-}}({}_-,\Bbbk):\mathrm{mod}\text{-}A\to  A\text{-}\mathrm{mod}.
\end{displaymath}

Let $L_1,L_2,\dots,L_n$ be a complete and irredundant list of representatives of isomorphism classes
of simple $A$-modules. Then $L^{*}_1,L^{*}_2,\dots,L^{*}_n$ is a complete and irredundant list of 
representatives of isomorphism classes of simple right $A$-modules. For $i,j=1,2,\dots,n$, set
$L_{ij}:=L_i\otimes_{\Bbbk}L^{*}_j$. This gives a complete and irredundant list of 
representatives of isomorphism classes of simple $A\text{-}A$--bimodules.
For $i=1,2,\dots,n$ we denote by $P_i$ and $I_i$ the indecomposable projective cover and injective envelope
of $L_i$, respectively.

When working with the opposite algebra, we will add the superscript ${}^{\mathrm{op}}$ to all notation.

We refer the reader to \cite{ARS,Ba,DK,GR} for further generalities and details on representation theory
of finite dimensional algebras.

\subsection{Trace functors}\label{s2.2}

With each $N\in A\text{-}\mathrm{mod}$ one associates the corresponding {\em trace functor}
$\mathrm{Tr}_N:A\text{-}\mathrm{mod}\to A\text{-}\mathrm{mod}$ defined in the following way:
\begin{itemize}
\item For every $M\in A\text{-}\mathrm{mod}$, the module $\mathrm{Tr}_N(M)\in A\text{-}\mathrm{mod}$ 
is defined as the submodule $\displaystyle\sum_{f:N\to M}\mathrm{Im}(f)$ of $M$. 
\item For every $M,M'\in A\text{-}\mathrm{mod}$ and every $f:M\to M'$, the corresponding morphism 
$\mathrm{Tr}_N(f):\mathrm{Tr}_N(M)\to \mathrm{Tr}_N(M')$ is defined as the restriction of $f$ to 
$\mathrm{Tr}_N(M)$. 
\end{itemize}
Directly from the definition it follows that $\mathrm{Tr}_N$ is a subfunctor of the identity functor for every
$N$. We denote by $\iota_{N}:\mathrm{Tr}_N\hookrightarrow \mathrm{Id}_{A\text{-}\mathrm{mod}}$ the
corresponding injective natural transformation.

\begin{lemma}\label{lem1}
Let $N\in A\text{-}\mathrm{mod}$.
\begin{enumerate}[$($i$)$]
\item\label{lem1.1} The functor $\mathrm{Tr}_N$ preserves monomorphisms.
\item\label{lem1.2} If $N$ is projective, then $\mathrm{Tr}_N$ preserves epimorphisms.
\item\label{lem1.3} We have  $\mathrm{Tr}_N\circ \mathrm{Tr}_N\cong \mathrm{Tr}_N$.
\end{enumerate}
\end{lemma}

\begin{proof}
Let $f:M\to M'$ be a monomorphism. In the commutative diagram
\begin{displaymath}
\xymatrix{
M\ar@{^{(}->}[rr]^f&&M'\\
\mathrm{Tr}_N(M)\ar[rr]^{\mathrm{Tr}_N(f)}\ar@{^{(}->}[u]^{\iota_M}&&\mathrm{Tr}_N(M')\ar@{^{(}->}[u]_{\iota_{M'}}
}
\end{displaymath}
we have $f\circ \iota_M$ is a monomorphism. Hence $\mathrm{Tr}_N(f)$ is a monomorphism as well. This proves 
claim~\eqref{lem1.1}.

Let $f:M\to M'$ be an epimorphism and $g:N\to M'$ any map. If $N$ is projective, then there is $h:N\to M$ such that 
$g=f\circ h$. Hence $\mathrm{Im}(g)=f(\mathrm{Im}(h))$ showing that $\mathrm{Tr}_N(M)$ surjects onto
$\mathrm{Tr}_N(M')$. This proves claim~\eqref{lem1.2}. 

Claim~\eqref{lem1.3} follows directly from the definition of $\mathrm{Tr}_N$. This completes the proof of the lemma.
\end{proof}

\begin{example}\label{nnewex1}
{\rm
In general, $\mathrm{Tr}_N$ is neither left nor right exact (even if $N$ is projective). 
Indeed, let $A$ be the
path algebra of the quiver $\xymatrix{1\ar[r]&2}$, $P_1$ be the indecomposable projective $A$-module
$\xymatrix{\Bbbk\ar[r]^{\mathrm{Id}}&\Bbbk}$, $L_1$ be the simple $A$-module
$\xymatrix{\Bbbk\ar[r]&0}$ and  $L_2$ be the simple $A$-module
$\xymatrix{0\ar[r]&\Bbbk}$. For $N=P_1$, applying $\mathrm{Tr}_N$ to the short exact sequence
\begin{displaymath}
0\to L_2 \to P_1\to L_1\to 0,
\end{displaymath}
gives the sequence
\begin{displaymath}
0\to  0 \to P_1\to L_1\to 0
\end{displaymath}
which has homology in the middle position.
}
\end{example}

\subsection{Projection functors}\label{s2.3}

For $N\in A\text{-}\mathrm{mod}$ we define the corresponding  {\em projection functor}
$\mathrm{Pr}_N:A\text{-}\mathrm{mod}\to A\text{-}\mathrm{mod}$ as the cokernel of the natural transformation
$\iota_{N}$. Let $\pi_{N}:\mathrm{Id}_{A\text{-}\mathrm{mod}}\tto \mathrm{Pr}_N$ denote the
corresponding surjective natural transformation.  The following properties of projection functors
appear in \cite{Pa,Gr}:
\begin{itemize}
\item For any  $N$, the functor $\mathrm{Pr}_N$ preserves epimorphisms.
\item If  $N$ is simple, then the functor $\mathrm{Pr}_N$ preserves monomorphisms.
\item If $N$ is simple and $\mathrm{Ext}_A^1(N,N)=0$, then $\mathrm{Pr}_N\circ \mathrm{Pr}_N\cong \mathrm{Pr}_N$.
\item If $N$ and $K$ are simple and $\mathrm{Ext}_A^1(K,N)=0$, then 
\begin{displaymath}
\mathrm{Pr}_N\circ \mathrm{Pr}_K\circ \mathrm{Pr}_N\cong 
\mathrm{Pr}_K\circ \mathrm{Pr}_N\circ \mathrm{Pr}_K\cong
\mathrm{Pr}_N  \circ \mathrm{Pr}_K.
\end{displaymath}
\item If $N$ and $K$ are simple and $\mathrm{Ext}_A^1(N,K)=\mathrm{Ext}_A^1(K,N)=0$, then 
\begin{displaymath}
\mathrm{Pr}_N\circ \mathrm{Pr}_K\cong \mathrm{Pr}_K\circ \mathrm{Pr}_N.
\end{displaymath}
\end{itemize}
For the record, we also point out the following connection between the functors $\mathrm{Pr}_N$ and $\mathrm{Tr}_N$.

\begin{lemma}\label{lem2}
For any fixed $N$, the functor $\mathrm{Tr}_N$ is exact if and only if the functor $\mathrm{Pr}_N$ is exact.
\end{lemma}

\begin{proof}
For an exact sequence  $X\overset{f}{\hookrightarrow} Y\overset{g}\tto Z$ in 
$A\text{-}\mathrm{mod}$ consider the commutative diagram 
\begin{displaymath}
\xymatrix{
\mathrm{Tr}_N(X)\ar@{^{(}->}[d]_{\iota_N(X)}\ar[rr]^{\mathrm{Tr}_N(f)}&&
\mathrm{Tr}_N(Y)\ar@{^{(}->}[d]_{\iota_N(Y)}\ar[rr]^{\mathrm{Tr}_N(g)}&&
\mathrm{Tr}_N(Z)\ar@{^{(}->}[d]^{\iota_N(Z)}\\
X\ar@{^{(}->}[rr]^f\ar@{->>}[d]_{\pi_N(X)}&&
Y\ar@{->>}[rr]^g\ar@{->>}[d]_{\pi_N(Y)}&&Z\ar@{->>}[d]^{\pi_N(Z)}\\
\mathrm{Pr}_N(X)\ar[rr]^{\mathrm{Pr}_N(f)}&&
\mathrm{Pr}_N(Y)\ar[rr]^{\mathrm{Pr}_N(g)}&&
\mathrm{Pr}_N(Z)\\
}
\end{displaymath}
Here all columns are exact by construction and the middle row is exact by assumption. Therefore the Nine Lemma
(a.k.a. the $3\times 3$-Lemma)
says that the first row is exact if and only if the third row is exact. 
\end{proof}

\section{Dual projection functors}\label{s3}

\subsection{Idealization functors}\label{s3.1}

The algebra $A$ is an $A\text{-}A$--bimodule, as usual.
Tensoring with this bimodule (over $A$) is isomorphic to the identity endofunctor of $A\text{-}\mathrm{mod}$.
We identify subbimodules of ${}_AA_A$ and
two-sided ideals of $A$. For each two-sided ideal $I\subset A$ denote by $\mathrm{Su}_I$ the endofunctor of
$A\text{-}\mathrm{mod}$ defined in the following way:
\begin{itemize}
\item For every $M\in A\text{-}\mathrm{mod}$, the module $\mathrm{Su}_I(M)$ is defined as $IM$.
\item For every $M,M'\in A\text{-}\mathrm{mod}$ and $f:M\to M'$, the morphism $\mathrm{Su}_I(f)$ is defined as 
the restriction of $f$ to $IM$.
\end{itemize}
We will call $\mathrm{Su}_I$ the {\em idealization functor} associated to $I$, where the notation $\mathrm{Su}$
stands for ``Sub''.

Let $\gamma:\mathrm{Su}_I\hookrightarrow \mathrm{Id}_{A\text{-}\mathrm{mod}}$ denote the injective natural
transformation given by the canonical inclusion $IM\hookrightarrow M$. Directly from the definition we obtain 
that for any two two-sided ideals $I$ and $J$ in $A$ we have
\begin{equation}\label{eq1}
\mathrm{Su}_I\circ \mathrm{Su}_J=\mathrm{Su}_{IJ}.
\end{equation}

Furthermore, if $I\subset J$, then we have the canonical inclusion $\mathrm{Su}_I\hookrightarrow \mathrm{Su}_J$.

\subsection{Exactness of idealization}\label{s3.2}

Here we prove the following property of idealization functors.

\begin{lemma}\label{lem3}
Let $I$ be a two-sided ideal in $A$.
\begin{enumerate}[$($i$)$]
\item\label{lem3.1} The functor $\mathrm{Su}_I$ preserves monomorphisms.
\item\label{lem3.2} The functor $\mathrm{Su}_I$ preserves epimorphisms.
\end{enumerate}
\end{lemma}

\begin{proof}
Claim~\eqref{lem3.1} follows from the definition of $\mathrm{Su}_I$ and the fact that the restriction of a
monomorphism is a monomorphism. To prove claim~\eqref{lem3.2}, consider an epimorphism $f:M\tto M'$,
$v\in M'$ and $a\in I$. Then there is $w\in M$ such that $f(w)=v$ and hence $af(w)=f(aw)=av$. 
As $aw\in \mathrm{Su}_I(M)$, we obtain that $av$ belongs to the image of $\mathrm{Su}_I(f)$, completing the proof.
\end{proof}

\begin{example}\label{nnewex2}
{\rm  
The functor $\mathrm{Su}_I$ is neither left nor right exact in general. Indeed, consider the algebra 
$A=\Bbbk[x]/(x^2)$, let $L$ be the (unique up to isomorphism) simple $A$-module and set $I:=\mathrm{Rad}(A)$. 
Applying $\mathrm{Su}_I$ to the short exact sequence
\begin{displaymath}
0\to  L\to {}_A A\to L\to 0,
\end{displaymath}
we obtain the sequence
\begin{displaymath}
0\to 0\to L\to 0\to 0 
\end{displaymath}
which has homology in the middle position.
}
\end{example}

\subsection{Idealization functors versus trace functors}\label{s3.3}

Let $I$ be a two-sided ideal of $A$ and $N$ an $A$-module. Since both $\mathrm{Su}_I$ and $\mathrm{Tr}_N$ 
are subfunctors of the identity functor, it is natural to ask when they are isomorphic. In this subsection 
we would like to present some examples showing that, in general, these two families of functors are really different.

\begin{lemma}\label{lem4}
If $A$ is not semi-simple, then $\mathrm{Su}_{\mathrm{Rad}(A)}$ is not isomorphic to any trace functor. 
\end{lemma}

\begin{proof}
We have $\mathrm{Su}_{\mathrm{Rad}(A)}(A)=\mathrm{Rad}(A)\neq 0$ as $A$ is not semi-simple. Hence
$\mathrm{Su}_{\mathrm{Rad}(A)}$ is not the zero functor, in particular, it is not isomorphic to $\mathrm{Tr}_0$.
At the same time, let $L:=A/\mathrm{Rad}(A)$. Then $\mathrm{Su}_{\mathrm{Rad}(A)}(L)=0$. 
On the other hand, for any non-zero $N\in A\text{-}\mathrm{mod}$ the module $N$ surjects onto some
simple $A$-module. As every simple $A$-module is a summand of $L$, we have
$\mathrm{Tr}_{N}(L)\neq 0$. The claim follows.
\end{proof}

\begin{lemma}\label{lem5}
If $N$ is simple and not projective, then $\mathrm{Tr}_{N}$ is not isomorphic to any idealization functor. 
\end{lemma}

\begin{proof}
Let $f:P\tto N$ be a projective cover of $N$. Then $P$ is indecomposable and has simple top.
As $N$ is simple, $\mathrm{Tr}_{N}(P)$ belongs to the socle of $P$. As $N$ is not projective,
$P\not\cong N$. Consequently, the socle of $P$ belongs to the radical of $P$. Therefore
$f$ annihilates $\mathrm{Tr}_{N}(P)$ and it follows
that $\mathrm{Tr}_{N}(f)$ is the zero map.
We also have the obvious isomorphism $\mathrm{Tr}_{N}(N)\cong N$.
At the same time, each idealization functor preserves epimorphisms by Lemma~\ref{lem3}\eqref{lem3.2}.
The obtained contradiction proves the statement.
\end{proof}

\subsection{Definition of dual projection functors}\label{s3.4}

Recall that, for any additive functor $\mathrm{F}:A\text{-}\mathrm{proj}\to A\text{-}\mathrm{mod}$, there is
a unique, up to isomorphism, right exact functor $\mathrm{G}:A\text{-}\mathrm{mod}\to A\text{-}\mathrm{mod}$
such that the restriction of $\mathrm{G}$ to $A\text{-}\mathrm{proj}$ is isomorphic to $\mathrm{F}$. 
As ${}_A A$ is an additive generator of $A\text{-}\mathrm{proj}$, the condition that 
the restriction of $\mathrm{G}$ to $A\text{-}\mathrm{proj}$ is isomorphic to $\mathrm{F}$
is equivalent to the condition that the $A\text{-}A$--bimodules 
$\mathrm{F}(A)$ and $\mathrm{G}(A)$ are isomorphic. The functor $\mathrm{G}$ is isomorphic to the functor
$\mathrm{F}(A)\otimes_A{}_-$, see \cite[Chapter~2,~{\S}2]{Ba} for details.

For an ideal $I$ in $A$ define a {\em dual projection functor} corresponding to $I$ as a 
functor isomorphic to the functor
\begin{displaymath}
\mathrm{Dp}_I:= \mathrm{Su}_I(A)\otimes_A{}_-:A\text{-}\mathrm{mod}\to A\text{-}\mathrm{mod}.
\end{displaymath}
Directly from the definition we have that $\mathrm{Dp}_I$ is right exact.

\begin{lemma}\label{lem6}
If $A$ is hereditary, then the functor $\mathrm{Dp}_I$ is exact for any $I$.
\end{lemma}

\begin{proof}
As $\mathrm{Su}_I(A)\subset A$ and $A$ is hereditary, the right $A$-module $\mathrm{Su}_I(A)$ is projective.
This means that $\mathrm{Dp}_I$ is exact.
\end{proof}

\begin{corollary}\label{cor7}
If $A$ is hereditary, then $\mathrm{Dp}_I\circ \mathrm{Dp}_J\cong \mathrm{Dp}_{IJ}$ for any
two two-sided ideals $I,J$ in $A$.
\end{corollary}

\begin{proof}
Note that for hereditary $A$ the functor $\mathrm{Dp}_I$ preserves $A\text{-}\mathrm{proj}$.
Because of exactness, established in Lemma~\ref{lem6}, it is thus enough to prove the isomorphism when 
restricted to   $A\text{-}\mathrm{proj}$ where it reduces to formula~\eqref{eq1}.
\end{proof}

\subsection{Special dual projection functors}\label{s3.5}

The radical of $A$ (see e.g. \cite[Section~3]{DK}) coincides with the radical of 
the $A\text{-}A$--bimodule ${}_AA_A$ and we have a short exact sequence
\begin{displaymath}
0\to \mathrm{Rad}(A)\to  {}_AA_A\to \bigoplus_{i=1}^n L_{ii}\to 0.
\end{displaymath}
For every $i=1,2,\dots,n$, this gives, using canonical projection onto a component of a direct sum, an
epimorphism ${}_AA_A\tto L_{ii}$. Let $J_i$ denote the kernel of the latter epimorphism. We will use
the shortcut $\mathrm{F}_i$ for the corresponding dual projection functor $\mathrm{Dp}_{J_i}$.
Setting $n_i:=\dim(L_i)$ for $i=1,2,\dots,n$, we have an isomorphism of left $A$-modules as follows:
\begin{equation}\label{eqan1}
J_i\cong \mathrm{Rad}(P_i)^{\oplus n_i}\oplus \bigoplus_{j\neq i} P_j^{\oplus n_j}. 
\end{equation}

\subsection{Dual projection functors versus projection functors}\label{s3.6}

In this subsection we explain the name {\em dual projection functors}.

For $i=1,2,\dots,n$ denote by $\mathrm{G}_i$ the unique, up to isomorphism, left exact endofunctor of
$A\text{-}\mathrm{mod}$ satisfying the condition that 
\begin{displaymath}
\mathrm{G}_i\vert_{A\text{-}\mathrm{inj}}\cong  \mathrm{Pr}_{L_i}\vert_{A\text{-}\mathrm{inj}}.
\end{displaymath}
For example, we can take
\begin{displaymath}
\mathrm{G}_i=\mathrm{Hom}_A((\mathrm{Pr}_{L_i}(A^*))^*,{}_-), 
\end{displaymath}
where $(\mathrm{Pr}_{L_i}(A^*))^*$ is viewed as an $A$--$A$-bimodule in the obvious way, 
see \cite[Subsection~2.3]{GrMa} for details. In other words, $\mathrm{G}_i$ is the unique 
left exact extension  of the projection functor corresponding to the simple module $L_i$.

\begin{proposition}\label{prop8}
There is an isomorphism of functors as follows:  $\mathrm{F}_i\cong *\circ \mathrm{G}^{\mathrm{op}}_i\circ *$.
\end{proposition}

\begin{proof}
Both $\mathrm{F}_i$ and $*\circ \mathrm{G}^{\mathrm{op}}_i\circ *$ are right exact functors and hence it is sufficient
to prove that they are isomorphic on $A\text{-}\mathrm{proj}$. For the additive generator $A$ of the
latter category we have 
\begin{displaymath}
(\mathrm{G}_i^{\mathrm{op}}(A^*))^*\cong 
\mathrm{Hom}_{\text{-}A}\big((\mathrm{Pr}^{\mathrm{op}}_{L^*_i}(A^*))^*,A^*\big)^*
\cong \mathrm{Hom}_{\text{-}\Bbbk}\big((\mathrm{Pr}^{\mathrm{op}}_{L^*_i}(A^*))^*,\Bbbk\big)^*\cong
(\mathrm{Pr}^{\mathrm{op}}_{L^*_i}(A^*))^*,
\end{displaymath}
where the second isomorphism is given by adjunction,
and thus the claim of our proposition amounts to finding a natural isomorphism between
$(\mathrm{F}_i(A))^*\cong J_i^*$ and $\mathrm{Pr}^{\mathrm{op}}_{L^*_i}(A^*)$.

Applying $*$ to the exact sequence $J_i\hookrightarrow A\tto L_{ii}$ results in the exact sequence
$L_{ii}^*\hookrightarrow A^*\tto J_i^*$. As $L^*_{ii}\cong L_{ii}$ and all other simple subbimodules of
$A^*$ are of the form $L_{jj}$ for some $j\neq i$, the submodule $\mathrm{Tr}^{\mathrm{op}}_{L^*_i}(A^*)$ coincides 
with $L^*_{ii}$. This implies that there is a bimodule isomorphism $J_i^*\cong \mathrm{Pr}^{\mathrm{op}}_{L^*_i}(A^*)$ 
which completes the proof.
\end{proof}

Proposition~\ref{prop8} allows us to freely transfer results for projection functors to dual
projection functors and vice versa. For technical reasons in this paper we will mostly work
with dual projection functors.

\subsection{Dual projection functors and coapproximation functors}\label{s3.7}

In some cases dual projective functors can be interpreted as partial coapproximation functors in the 
terminology of \cite[Subsection~2.4]{KhMa}. For $i=1,2,\dots,n$, set 
\begin{displaymath}
Q_i:=P_1\oplus P_2\oplus \dots\oplus P_{i-2}\oplus P_{i-1}\oplus P_{i+1}\oplus P_{i+2}\oplus \dots \oplus 
P_{n-1}\oplus P_n. 
\end{displaymath}
The functor $\mathrm{C}_i$ of {\em partial coapproximation} with respect to $Q_i$ is defined as follows:
Given $M\in A\text{-}\mathrm{mod}$, consider a short exact sequence $K\hookrightarrow P\tto M$ with 
projective $P$. Then 
\begin{displaymath}
\mathrm{C}_i(M):=\mathrm{Tr}_{Q_i}(P/\mathrm{Tr}_{Q_i}(K))
\end{displaymath}
and the action on morphisms is defined by first lifting them using projectivity and then restriction.
From \cite[Lemma~9]{KhMa} it follows that $\mathrm{C}_i$ is right exact. The functor $\mathrm{C}_i$ comes
together with a natural transformation $\kappa:\mathrm{C}_i\to\mathrm{Id}_{A\text{-}\mathrm{mod}}$
which is injective on projective modules (note that, if $M$ is projective in the above construction,
then we may choose $K=0$ and $\mathrm{C}_i(M)=\mathrm{Tr}_{Q_i}(M)$). In particular, if 
$\mathrm{Ext}_A^1(L_i,L_i)=0$,  then we have
\begin{displaymath}
\mathrm{Tr}_{Q_i}(P_j)\cong
\begin{cases} 
P_j, \quad &\text{ if } i\neq j;\\ 
\mathrm{Rad}(P_i), &\text{otherwise}.
\end{cases}
\end{displaymath}

\begin{lemma}\label{lem9}
If $\mathrm{Ext}_A^1(L_i,L_i)=0$, then $\mathrm{C}_i\cong \mathrm{F}_i$.
\end{lemma}

\begin{proof}
As both functors are right exact, it is enough to check the bimodule isomorphism 
$\mathrm{C}_i(A)\cong \mathrm{F}_i(A)$. Since ${}_AA$ is projective, we have 
$\mathrm{C}_i(A)=\mathrm{Tr}_{Q_i}(A)$. At the same time, if $\mathrm{Ext}_A^1(L_i,L_i)=0$, then
$\mathrm{Tr}_{Q_i}(A)=J_i$. As the action of $\mathrm{C}_i$  on morphisms is defined
via restriction, it follows that $\mathrm{C}_i(A)\cong J_i$ as a bimodule. This completes the proof.
\end{proof}

\section{Some preliminaries on $2$-categories}\label{s4}

\subsection{Finite and finitary $2$-categories}\label{s4.1}

We refer the reader to \cite{Le,McL,Ma} for generalities on $2$-categories. Denote by $\mathbf{Cat}$ the category
of all small categories. A {\em $2$-category} is a category enriched over $\mathbf{Cat}$. A $2$-category $\cC$ is 
called {\em finite} if it has finitely many objects, finitely many $1$-morphisms and finitely many $2$-morphisms.

Recall from \cite{MM1}  that a $2$-category $\cC$ is called {\em finitary} over $\Bbbk$ provided that
\begin{itemize}
\item $\cC$ has finitely many objects;
\item each $\cC(\mathtt{i},\mathtt{j})$ is an idempotent split additive $\Bbbk$-linear category with finitely
many isomorphism classes of indecomposable objects and finite dimensional spaces of morphisms;
\item all compositions are biadditive and also $\Bbbk$-bilinear whenever the latter makes sense;
\item all identity $1$-morphisms are indecomposable.
\end{itemize}

For an object ${\mathtt{i}}$ of a $2$-category we denote by $\mathbbm{1}_{\mathtt{i}}$ the corresponding
identity $1$-morphism.

\subsection{The multisemigroup of a finitary $2$-category}\label{s4.2}

For a finitary $2$-category $\cC$ denote by $\mathcal{S}_{\ccC}$ the set of isomorphism classes of 
indecomposable $1$-morphisms in $\cC$ with an added external zero element $0$. 
By \cite[Subsection~3.3]{MM2}, the finite set 
$\mathcal{S}_{\ccC}$ has the natural structure of a {\em multisemigroup} given for 
$[\mathrm{F}],[\mathrm{G}]\in \mathcal{S}_{\ccC}$ by defining
\begin{displaymath}
[\mathrm{F}]\star[\mathrm{G}]:=
\begin{cases}
\left\{[\mathrm{H}]\,:\, \mathrm{H}\text{ is isomorphic to a direct summand of }
\mathrm{F}\circ\mathrm{G}\right\}, & \mathrm{F}\circ\mathrm{G}\neq 0;\\
0, & \text{otherwise}.
\end{cases}
\end{displaymath}
We refer the reader to \cite{KuMa} for more details on multisemigroups.

\subsection{$\Bbbk$-linearization of finite categories}\label{s4.3}

For a set $X$ denote by $\Bbbk[X]$ the $\Bbbk$-vector space of all formal linear combinations of
elements in $X$ with coefficients in $\Bbbk$. Then $X$ is naturally identified with a basis in $\Bbbk[X]$.
Note that $\Bbbk[X]=\{0\}$ if $X=\varnothing$.

Let $\mathcal{C}$ be a finite category, that is a category with finitely many objects and morphisms. 
The {\em $\Bbbk$-linearization} of $\mathcal{C}$ is the category $\mathcal{C}_{\Bbbk}$ defined as follows:
\begin{itemize}
\item $\mathcal{C}_{\Bbbk}$ and $\mathcal{C}$ have the same objects;
\item $\mathcal{C}_{\Bbbk}(\mathtt{i},\mathtt{j}):=\Bbbk[\mathcal{C}(\mathtt{i},\mathtt{j})]$;
\item composition in $\mathcal{C}_{\Bbbk}$ is induced from composition in $\mathcal{C}$ by $\Bbbk$-bilinearity.
\end{itemize}
The {\em additive $\Bbbk$-linearization} $\mathcal{C}_{\Bbbk}^{\oplus}$ of $\mathcal{C}$ is 
then the ``additive closure'' of $\mathcal{C}_{\Bbbk}$ in the following sense:
\begin{itemize}
\item objects in $\mathcal{C}_{\Bbbk}^{\oplus}$ are all expressions of the form $\mathtt{i}_1\oplus
\mathtt{i}_2\oplus\dots\oplus\mathtt{i}_k$, where $k\in\{0,1,2,\dots\}$ and all 
$\mathtt{i}_i$ are objects in $\mathcal{C}_{\Bbbk}$;
\item the set  $\mathcal{C}_{\Bbbk}^{\oplus}(\mathtt{i}_1\oplus
\mathtt{i}_2\oplus\dots\oplus\mathtt{i}_k,\mathtt{j}_1\oplus
\mathtt{j}_2\oplus\dots\oplus\mathtt{j}_m)$ consists of all matrices of the form
\begin{displaymath}
\left(
\begin{array}{cccc}
f_{11}& f_{12}&\dots& f_{1k} \\
f_{21}& f_{22}&\dots& f_{2k} \\
\vdots& \vdots&\ddots& \vdots \\
f_{m1}& f_{m2}&\dots& f_{mk}  
\end{array}
\right) 
\end{displaymath}
where $f_{st}\in\mathcal{C}_{\Bbbk}(\mathtt{i}_t,\mathtt{j}_s)$;
\item composition in $\mathcal{C}_{\Bbbk}^{\oplus}$ is given by matrix multiplication.
\end{itemize}

\subsection{Finitarization of finite $2$-categories}\label{s4.4}

Let $\cC$ be a finite $2$-category. Then the {\em finitarization} of $\cC$ over $\Bbbk$ is the $2$-category
$\cC_{\Bbbk}$ defined as follows:
\begin{itemize}
\item $\cC_{\Bbbk}$ has the same objects as $\cC$;
\item $\cC_{\Bbbk}(\mathtt{i},\mathtt{j}):=\cC(\mathtt{i},\mathtt{j})_{\Bbbk}^{\oplus}$;
\item composition in $\cC_{\Bbbk}$ is induced from composition in $\cC$ using biadditivity and $\Bbbk$-bilinearity.
\end{itemize}
Directly from the definition it follows that $\cC_{\Bbbk}$ is finitary if and only if for each
$1$-morphism $f$ in $\cC$ the endomorphism algebra
$\mathrm{End}_{\ccC_{\Bbbk}}(f) \cong \Bbbk[\mathrm{End}_{\ccC}(f)]$ is local.

\subsection{Two $2$-categories associated with an ordered monoid}\label{s4.5}

Let $(S,e,\cdot)$ be a finite monoid with a fixed admissible reflexive partial pre-order $\preceq$. Admissibility 
means that $s\preceq t$ implies both $sr\preceq tr$ and $rs\preceq rt$ for all $s,t,r\in S$. In this situation we
may define a finite $2$-category $\cC^{S}$ as follows:
\begin{itemize}
\item $\cC^{S}$ has one object $\clubsuit$;
\item $1$-morphisms in $\cC^{S}(\clubsuit,\clubsuit)$ are elements in $S$ and the horizontal 
composition of $1$-morphisms is given by multiplication in $S$;
\item for two $1$-morphisms $s$ and $t$, the set of $2$-morphisms from $s$ to $t$ is empty if $s\not\preceq t$
and contains one element, denoted $(s,t)$, otherwise (note that in this case all compositions of $2$-morphisms
are automatically uniquely defined).
\end{itemize}
The finitarization $\cC^S_{\Bbbk}$ of $\cC^{S}$ is then a finitary $2$-category as the endomorphism algebra
of each $1$-morphism is just $\Bbbk$.

\section{$2$-categories of idealization functors and dual projection functors}\label{s5}

\subsection{Monoid of two-sided ideals}\label{s5.1}

The set $\mathcal{I}$ of all two-sided ideals in $A$ has the natural structure of a monoid given by
multiplication of ideals $(I,J)\mapsto IJ$. The identity element of $\mathcal{I}$ is $A$ and the zero
element is the zero ideal. We note the following:

\begin{lemma}\label{lem10}
If $\dim_{\Bbbk}\mathrm{Hom}_A(P_i,P_j)\leq 1$ for all $i,j\in\{1,2,\dots,n\}$, then
$|\mathcal{I}|<\infty$.
\end{lemma}

\begin{proof}
If $a,b\in A$ are idempotents, then, by adjunction, we have
\begin{displaymath}
\mathrm{Hom}_{A\text{-}A}(Aa\otimes_{\Bbbk}bA,A)\cong
\mathrm{Hom}_{A\text{-}}(Aa,Ab)\cong
\mathrm{Hom}_{\Bbbk}(\Bbbk,aAb)=aAb.
\end{displaymath}
For $i,j\in\{1,2,\dots,n\}$, the projective cover of the simple bimodule
$L_{ij}$ in $A\text{-}\mathrm{mod}\text{-}A$ is isomorphic
to $P_i\otimes_{\Bbbk}I_j^*$ and hence from our assumptions it follows that the composition
multiplicity of $L_{ij}$ in ${}_AA_A$ is at most $1$. This means that each subbimodule of ${}_AA_A$ is 
uniquely determined by its composition subquotients (and equals the sum of images of unique up to scalar
nonzero homomorphisms from the projective covers of these simple subquotients). Therefore 
$|\mathcal{I}|\leq 2^{\dim(A)}$.
\end{proof}

\begin{corollary}\label{cor11}
If $A$ is the path algebra of a tree quiver
or the incidence algebra of a finite poset, then $|\mathcal{I}|<\infty$. 
\end{corollary}

\begin{proof}
Both for the path algebra of a tree quiver and for the incidence algebra of a finite poset, the condition 
$\dim_{\Bbbk}\mathrm{Hom}_A(P_i,P_j)\leq 1$ for all $i,j\in\{1,2,\dots,n\}$
is straightforward and thus the statement follows from Lemma~\ref{lem10}.
\end{proof}

The monoid $\mathcal{I}$ is naturally ordered by inclusions, moreover, this order is obviously admissible.

\subsection{A $2$-action of $\cC^{\mathcal{I}}$ on $A\text{-}\mathrm{mod}$ by idealization functors}\label{s5.2}

We define a $2$-action of the $2$-category $\cC^{\mathcal{I}}$ associated to the ordered monoid 
$(\mathcal{I},A,\cdot,\subset)$ on $A\text{-}\mathrm{mod}$ as follows:
\begin{itemize}
\item the element $I\in \mathcal{I}$ acts as the functor $\mathrm{Su}_I$;
\item for $I\subset J$, the $2$-morphism $(I,J)$ acts as the canonical inclusion
$\mathrm{Su}_I\hookrightarrow \mathrm{Su}_J$.
\end{itemize}
This is a strict $2$-action because of \eqref{eq1}.

This $2$-action extends to a $2$-action of $\cC^{\mathcal{I}}_{\Bbbk}$ on $A\text{-}\mathrm{mod}$ in the
obvious way. Note that this $2$-action is clearly faithful both on the level of $1$-morphisms and on the
level of $2$-morphisms. However, this $2$-action is not full on the level of $2$-morphisms in general.
Indeed, in case the algebra $A$ has a non-trivial center, the $1$-dimensional endomorphism algebra of the
identity $1$-morphism in $\cC^{\mathcal{I}}_{\Bbbk}$ cannot surject onto the non-trivial endomorphism algebra
of the identity functor of $A\text{-}\mathrm{mod}$.

\subsection{A $2$-action of $\cC^{\mathcal{I}}$ on $A\text{-}\mathrm{mod}$ by dual projection functors}\label{s5.3}

The main disadvantage of the $2$-action defined in Subsection~\ref{s5.2} is the fact that the functors
$\mathrm{Su}_J$ are not exact from any side in general. In particular, they do not induce any reasonable
maps on the Grothendieck group of $A\text{-}\mathrm{mod}$. To overcome this problem one needs to define another
action and dual projection functors are reasonable candidates. However, there is a price to pay. 
Firstly, in order to avoid weak $2$-actions (where equalities of functors are changed 
to isomorphisms with some coherency conditions, see e.g. \cite{Le}), 
we will have to change $A\text{-}\mathrm{mod}$ to an equivalent category.
Secondly, we will have to restrict to hereditary algebras.

Denote by $\overline{A\text{-}\mathrm{proj}}$ the category whose objects are diagrams 
$P\overset{f}{\longrightarrow}Q$ over $A\text{-}\mathrm{proj}$ and whose morphisms are equivalence
classes of solid commutative diagrams 
\begin{displaymath}
\xymatrix{
P\ar[rr]^{f}\ar[d]_{g_2}&&Q\ar[d]^{g_1}\ar@{-->}[dll]_{h}\\
P'\ar[rr]^{f'}&&Q'\\
}
\end{displaymath}
modulo the equivalence relation defined as follows: the solid diagram is equivalent to zero provided that there
exists a dashed map $h$ as indicated on the diagram such that $g_1=f'h$. The category 
$\overline{A\text{-}\mathrm{proj}}$ is abelian and, moreover, equivalent to $A\text{-}\mathrm{mod}$, see \cite{Fr}.
This construction is called {\em abelianization} in \cite{MM1,MM2}.

If $A$ is hereditary, then each $\mathrm{Su}_I$ preserves $A\text{-}\mathrm{proj}$ and hence the $2$-actions of both 
$\cC^{\mathcal{I}}$ and $\cC^{\mathcal{I}}_{\Bbbk}$ defined in Subsection~\ref{s5.2} extends 
component-wise to $2$-actions of both these categories on $\overline{A\text{-}\mathrm{proj}}$.
By construction, this is not an action on $A\text{-}\mathrm{mod}$ but on a category which is only equivalent to
$A\text{-}\mathrm{mod}$. Moreover, the action is designed so that the ideal $I$ acts by a right exact functor
which is isomorphic to $\mathrm{Su}_I$ when restricted to $A\text{-}\mathrm{proj}$. This means that this is
a $2$-action by dual projection functors.

\subsection{The $2$-category of idealization functors}\label{s5.4}

The $2$-action defined in Subsection~\ref{s5.2} suggest the following definition. Fix a small category $\mathcal{C}$
equivalent to $A\text{-}\mathrm{mod}$. Define the $2$-category $\cQ=\cQ(A,\mathcal{C})$ in the following way:
\begin{itemize}
\item $\cQ$ has one object $\clubsuit$ (which we identify with $\mathcal{C}$);
\item $1$-morphisms in $\cQ$ are endofunctors of $\mathcal{C}$ which belong to the additive closure generated by 
the identity functor and all idealization functors;
\item $2$-morphisms in $\cQ$ are all natural transformations of functors;
\item composition in  $\cQ$ is induced from $\mathbf{Cat}$.
\end{itemize}
Our main observation here is the following:

\begin{proposition}\label{prop15}
If $A$ is connected and $|\mathcal{I}|<\infty$, then $\cQ$ is a finitary $2$-category.
\end{proposition}

\begin{proof}
Connectedness of $A$ ensures that the identity $1$-morphism $\mathbbm{1}_{\clubsuit}$ is indecomposable.
Clearly, $\cQ$ has finitely many objects. As $|\mathcal{I}|<\infty$, the $2$-category $\cQ$ 
has finitely many isomorphism classes of
indecomposable $1$-morphisms. It remains to check that all spaces of $2$-morphisms are finite dimensional.

Let $I$ and $J$ be two ideals in $A$. Let $\eta:\mathrm{Su}_I\to\mathrm{Su}_J$ be a natural transformation.
We claim that values of $\eta$ on indecomposable projective $A$-modules determine $\eta$ uniquely. 
Indeed, by additivity these values determine all values of $\eta$ on all projective $A$-modules.
For $M\in A\text{-}\mathrm{mod}$, choose some projective cover $f:P\tto M$. Then,
by Lemma~\ref{lem3}\eqref{lem3.2}, we have the commutative diagram:
\begin{displaymath}
\xymatrix{ 
\mathrm{Su}_I(P)\ar@{->>}[rr]^{\mathrm{Su}_I(f)}\ar[d]_{\eta_P}&&\mathrm{Su}_I(M)\ar[d]^{\eta_M}\\
\mathrm{Su}_J(P)\ar@{->>}[rr]^{\mathrm{Su}_J(f)}&&\mathrm{Su}_J(M)\\
}
\end{displaymath}
From this diagram we see that  $\eta_M$ is uniquely determined by $\eta_P$.
Consequently, all spaces of $2$-morphisms in $\cQ$ are finite dimensional.
\end{proof}

The fact that $\mathrm{Su}_I$ is not right exact implies that, potentially, there might exist a 
natural transformation $\eta\vert_{A\text{-}\mathrm{proj}}:
\mathrm{Su}_I\vert_{A\text{-}\mathrm{proj}}\to\mathrm{Su}_J\vert_{A\text{-}\mathrm{proj}}$ which cannot be
extended to a natural transformation $\eta:\mathrm{Su}_I\to\mathrm{Su}_J$.
Note also that in the case when $A$ has finite representation type the space of natural transformations 
between any two additive endofunctors on $A\text{-}\mathrm{mod}$ is finite dimensional (since an additive
endofunctor is uniquely determined, up to isomorphism, by its action on indecomposable objects and
morphisms between them and in the case when $A$ has finite representation type there are only finitely many 
indecomposable $A$-modules).

\subsection{The $2$-category of dual projection functors}\label{s5.5}

The $2$-action defined in Subsection~\ref{s5.3} suggest the following definition. Assume that $A$ is hereditary.
Fix a small category $\mathcal{C}$ equivalent to $A\text{-}\mathrm{mod}$. Define the $2$-category 
$\cP=\cP(A,\mathcal{C})$ in the following way:
\begin{itemize}
\item $\cP$ has one object $\clubsuit$ (which we identify with $\mathcal{C}$);
\item $1$-morphisms in $\cP$ are endofunctors of $\mathcal{C}$ which belong to the additive closure generated by 
the identity functor and all dual projection functors;
\item $2$-morphisms in $\cP$ are all natural transformations of functors;
\item composition in  $\cP$ is induced from $\mathbf{Cat}$.
\end{itemize}
Our main observation here is the following:

\begin{proposition}\label{prop16}
If $A$ is hereditary, connected and $|\mathcal{I}|<\infty$, then $\cP$ is a finitary $2$-category.
\end{proposition}

\begin{proof}
Similarly to the proof of Proposition~\ref{prop15}, the $2$-category $\cP$ has one object, finitely
many isomorphism classes of indecomposable $1$-morphisms thanks to the assumption
$|\mathcal{I}|<\infty$, and indecomposable identity $1$-morphism $\mathbbm{1}_{\clubsuit}$ thanks to the
assumption that $A$ is connected. Spaces of $2$-morphisms are finite dimensional as
projection functors are right exact and hence are given by tensoring with finite dimensional
bimodules which yields that spaces of $2$-morphisms are just bimodule homomorphisms 
between these finite dimensional bimodules.
\end{proof}

\subsection{Decategorification and categorification}\label{s5.6}

Let $\cC$ be a finitary $2$-category. Then the {\em decategorification} of $\cC$ is the ($1$-)category $[\cC]$
defined as follows.
\begin{itemize}
\item $[\cC]$ has same objects as $\cC$;
\item for all $\mathtt{i},\mathtt{j}\in\cC$ the morphism set $[\cC](\mathtt{i},\mathtt{j})$ is defined to be the
split Grothendieck group $[\cC(\mathtt{i},\mathtt{j})]_{\oplus}$ of the additive category $\cC(\mathtt{i},\mathtt{j})$;
\item composition in $[\cC]$ is induced from composition in $\cC$.
\end{itemize}

Given a $2$-functor $\Phi$ from $\cC$ to the $2$-category of additive categories, taking the split Grothendieck group
for each $\Phi(\mathtt{i})$ induces a functor $[\Phi]$ from $[\cC]$ to $\mathbf{Cat}$ which is called the
{\em decategorification} of $\Phi$. 

Given a $2$-functor $\Phi$ from $\cC$ to the $2$-category of abelian categories and exact functors, taking the
usual Grothendieck group for each $\Phi(\mathtt{i})$ induces a functor $[\Phi]$ from $[\cC]$ to $\mathbf{Cat}$
which is also called the {\em decategorification} of $\Phi$.

Conversely, the $2$-category $\cC$ is called a {\em categorification} of the category $[\cC]$ and the $2$-functor
$\Phi$ is called a {\em categorification} of the functor $\Phi$. We refer to \cite[Section~1]{Ma} 
for more details and examples.

\section{Indecomposable summands of dual projection functors for path algebras of admissible trees}\label{s7}

In this section we study both the monoid $\mathcal{I}$ and the multisemigroup $\mathcal{S}_{\ccP}$ 
in case $A$ is the path algebra of the quiver $Q$ given by an admissible orientation of a tree.

\subsection{Categorification of the Catalan monoid}\label{s7.1}

To start with,  we briefly recall the main results from \cite{GrMa}. Let $A$ be the path algebra of the 
following quiver
\begin{equation}\label{eq2}
\xymatrix{
1\ar[rr]&&2\ar[rr] &&3 \ar[rr]&& \dots\ar[rr]&& m.
}
\end{equation}
The main result of \cite{GrMa} asserts that the ring $[\cP](\clubsuit,\clubsuit)$ in the 
corresponding decategorification
is isomorphic to the integral monoid algebra of the monoid $\mathtt{C}_{m+1}$ of all order preserving and 
decreasing transformations of $\{1,2,\dots,m,m+1\}$, which is also known as the {\em Catalan monoid}. 
Moreover, the monoid 
$\mathtt{C}_{m+1}$ is an ordered monoid and the $2$-category $\cP$ is biequivalent to the corresponding $2$-category 
$\cC^{\mathtt{C}_{m+1}}_{\Bbbk}$. We can also observe that in this case the multisemigroup 
$\mathcal{S}_{\ccP}$ is a usual monoid (i.e. the operation in $\mathcal{S}_{\ccP}$ is single-valued rather than
multi-valued) and is, in fact, isomorphic to $\mathtt{C}_{m+1}$.

A very special feature of this example is the fact that the bimodule ${}_AA_A$ has simple socle. 
Consequently, all ideals of $A$ are indecomposable as $A\text{-}A$--bimodules. One observation in addition to
the results from \cite{GrMa} is the following.

\begin{proposition}\label{prop21}
If $A$ is the path algebra of the quiver \eqref{eq2}, then the $2$-categories 
$\cQ$ and $\cP$ are biequivalent.
\end{proposition}

\begin{proof}
Note that in this situation $A$ is hereditary and connected. From Corollary~\ref{cor11} it follows
that $|\mathcal{I}|<\infty$. In particular, both $\cQ$ and $\cP$ are well-defined and finitary,
see Propositions~\ref{prop15} and \ref{prop16}.
For both of these $2$-categories consider the restriction $2$-functor to the $2$-category 
of additive endofunctors on $\mathcal{C}_{\mathrm{proj}}$, where the latter stands for the category
of projective objects in $\mathcal{C}$. This is well defined as the action of $\mathrm{Su}_I$ preserves 
$\mathcal{C}_{\mathrm{proj}}$ for each $I$ as $A$ is hereditary. The restriction $2$-functor is clearly
faithful both on the level of $1$-morphisms and on the level of $2$-morphisms.

Now, for any  non-zero $I$ and $J$, the space $\mathrm{Hom}_{A\text{-}A}(I,J)$ is zero if $I\not\subset J$ and is
one-dimensional otherwise since both $I$ and $J$ have simple socle 
(as $A$--$A$-bimodules) and the corresponding simple bimodule appears with multiplicity one in both
of them. This means that $\mathrm{Hom}_{\ccP}(\mathrm{Dp}_I,\mathrm{Dp}_J)$ is zero if $I\not\subset J$ and is
one-dimensional otherwise. 

As the restrictions of $\mathrm{Su}_I$ and $\mathrm{Dp}_I$ to $\mathcal{C}_{\mathrm{proj}}$ are isomorphic
(by construction), from the previous paragraph and the proof of Proposition~\ref{prop15} it follows that 
$\mathrm{Hom}_{\ccQ}(\mathrm{Su}_I,\mathrm{Su}_J)$ is zero if $I\not\subset J$ and is at most
one-dimensional otherwise. However, the inclusion $I\subset J$ does give rise to a non-zero
natural transformation in $\mathrm{Hom}_{\ccQ}(\mathrm{Su}_I,\mathrm{Su}_J)$ in the obvious way. Therefore
$\mathrm{Hom}_{\ccQ}(\mathrm{Su}_I,\mathrm{Su}_J)$ is one-dimensional if $I\subset J$.
This implies that both restriction $2$-functors are full and faithful. As already noted above, 
by construction of dual projection functors, the values of both these restrictions hit
exactly the same isomorphism classes of endofunctors of $\mathcal{C}_{\mathrm{proj}}$.
The claim follows.
\end{proof}

\subsection{Setup and some combinatorics}\label{s7.2}

For a vertex $i$ of an oriented graph $\Gamma$ we denote by $\deg_{\Gamma}(i)$ the degree of $i$, by
$\deg_{\Gamma}^{\mathrm{in}}(i)$ the in-degree of $i$ and by $\deg_{\Gamma}^{\mathrm{out}}(i)$ 
the out-degree of $i$. Clearly, $\deg_{\Gamma}(i)=\deg_{\Gamma}^{\mathrm{in}}(i)+\deg_{\Gamma}^{\mathrm{out}}(i)$.

In the rest of the paper we consider an oriented (connected) tree $Q$ with vertex set
$Q_0=\{1,2,\dots,n\}$, where $n>1$. Set 
\begin{displaymath}
\mathbf{K}(Q)=\{i\in Q_0\,;\, \deg_{Q}^{\mathrm{in}}(i)\deg_{Q}^{\mathrm{out}}(i)=0\},\qquad 
\mathbf{K}'(Q)=\{i\in \mathbf{K}(Q)\,;\,\deg_{Q}(i)\geq 2\}.
\end{displaymath}
In other words,  $\mathbf{K}(Q)$ is the set of all sinks and sources in $Q$ and 
$\mathbf{K}'(Q)$ is the set of all elements $i\in\mathbf{K}(Q)$ which are not leaves.
In what follows identify subsets in $Q_0$ with the corresponding full
subgraphs in $Q$. A function $\alpha:Q_0\to X$, for any $X$, will be written
$\alpha=(\alpha(1),\alpha(2),\dots,\alpha(n))$.

Following \cite{Gr}, we say that $Q$ is {\em admissible} provided that all vertices of $Q$ of degree
at least $3$ belong to $\mathbf{K}(Q)$. 

\begin{example}\label{nnewex3}
{\rm  
The orientation of a $D_4$ diagram on the left hand side
of the following picture is admissible while the one on the right hand side is not.
\begin{equation}\label{eq101}
\xymatrix{
&4\ar[d]&\\
1\ar[r]&2&3\ar[l]
}\qquad\qquad\qquad
\xymatrix{
&4\ar[d]&\\
1\ar[r]&2\ar[r]&3
}
\end{equation}
}
\end{example}

For $i\in Q_0$, we denote by $\overline{i}$ 
the set of all elements in $Q_0$ to which there is
an oriented path (possibly empty) from $i$ in the quiver $Q$.
Elements in $\overline{i}$ will be called {\em successors} of $i$.
We have $j\in \overline{i}$ if and only if the pair $(j,i)$ belongs
to the transitive closure of the binary relation given by elements
in $Q_1$ (our convention is that the arrow from $s$ to $t$ corresponds 
to the pair $(t,s)$). For $X\subset Q_0$, we define
\begin{displaymath}
\overline{X}:=\bigcup_{x\in X} \overline{x}.
\end{displaymath}
Note that $\overline{\varnothing}=\varnothing$.

\begin{example}\label{nnewex4}
{\rm
For the left quiver in \eqref{eq101} we have $\overline{1}=\{1,2\}$ while for the right 
quiver in \eqref{eq101} we have $\overline{1}=\{1,2,3\}$ and $\overline{\{1,4\}}=Q_0$.
}
\end{example}

A function $\alpha:Q_0\to Q_0\cup\{0\}$ is called a {\em path function} provided that 
$\alpha(i)\in\overline{i}\cup\{0\}$ for all $i$. A path function is called {\em monotone} 
provided that, for all $i,j$ such that $i\in\overline{j}$ and $\alpha(i)\neq 0$, we have 
$\alpha(j)\neq0$ and $\alpha(i)\in\overline{\alpha(j)}$. In particular, a monotone 
function maps to zero all successors of any preimage of zero.

\begin{example}\label{nnewex5}
{\rm
The identity function on $Q_0$ is a monotone path functions.
The function which maps all elements of $Q_0$ to $0$ is a monotone path function.
At the same time, for the left quiver in \eqref{eq101}, the function which maps $1$ to $0$
and $i$ to $i$ for all $i=2,3,4$ is a path function but it is not monotone.
}
\end{example}

A {\em chain} in $Q$ is a subtree isomorphic to \eqref{eq2} for some $m$. The set of 
all  chains in $Q$ is partially ordered by inclusions. A {\em maximal chain} is a chain which 
is maximal with respect to this partial order. We denote by $\mathfrak{M}_Q$ the set of all maximal
chains in $Q$. 

\begin{example}\label{nnewex5-5}
{\rm
For the left quiver in \eqref{eq101}, we have
$\mathfrak{M}_Q=\{\{1,2\},\{2,3\},\{2,4\}\}$.
}
\end{example}

For a function $\alpha:Q_0\to Q_0\cup\{0\}$, define the {\em support} $\mathrm{supp}(\alpha)$
of $\alpha$  as the (not necessarily full) subgraph $Q^{(\alpha)}$ of $Q$ given by the
union of all $X\in \mathfrak{M}_Q$ which contain some $i\in Q_0$ such that $\alpha(i)\in X$.
If $\alpha$ is a path function, then $X\in \mathfrak{M}_Q$ belongs to  
$\mathrm{supp}(\alpha)$ if and only if the image of $\alpha$ intersects $X$.

\begin{example}\label{nnewex5-6}
{\rm
Consider the quiver $\xymatrix{1\ar[r]&2&3\ar[l]\ar[r]&4}$. For the monotone path
function $\alpha=(1,2,2,0)$, the graph $\mathrm{supp}(\alpha)$ is
\begin{displaymath}
\xymatrix{1\ar[r]&2&3\ar[l]},
\end{displaymath}
in particular, it is connected. For the monotone path function $\beta=(1,0,4,4)$,
the graph $\mathrm{supp}(\beta)$ is
\begin{displaymath}
\xymatrix{1\ar[r]&2&3\ar[r]&4},
\end{displaymath}
in particular, it is disconnected.  
}
\end{example}

A monotone path  function $\alpha:Q_0\to Q_0\cup\{0\}$ will be called a {\em special function} provided 
that its support is connected and, additionally, the equality  $\alpha(i)=0$ for 
$i\in \mathbf{K}(Q)\cap \mathrm{supp}(\alpha)$ implies that $\deg_{\mathrm{supp}(\alpha)}(i)=1$.

\begin{lemma}\label{supportlemma}
Let $Q$ be admissible and $\alpha$ a special function. Then
\begin{enumerate}[$($i$)$]
\item\label{usl2} for any $i\in \mathrm{supp}(\alpha)$ we have 
$\deg_{\mathrm{supp}(\alpha)}(i)\in\{\deg_{Q}(i),1\}$;
\item\label{usl3} $\alpha(i)\neq i$ for any $i\in \mathbf{K}(Q)\cap \mathrm{supp}(\alpha)$ 
with $\deg_{\mathrm{supp}(\alpha)}(i)=1$;
\item\label{usl4} $\alpha(i)=i$ for any $i\in \mathbf{K}(Q)\cap \mathrm{supp}(\alpha)$ 
with $\deg_{\mathrm{supp}(\alpha)}(i)=\deg_{Q}(i)$. 
\end{enumerate}
\end{lemma}

\begin{proof}
Let $i\in \mathrm{supp}(\alpha)$. First consider the case $i\in Q_0\setminus\mathbf{K}(Q)$. 
In this case $\mathrm{deg}_Q(i)=2$ since $Q$ is admissible.
Therefore there is a unique $X\in \mathfrak{M}_Q$ containing $i$, moreover, $i$ is not a leaf in $X$.
From the definitions we thus have $X\subset \mathrm{supp}(\alpha)$ and hence 
$\deg_{\mathrm{supp}(\alpha)}(i)=\mathrm{deg}_Q(i)=2$.

Assume now that $i\in\mathbf{K}(Q)$ is a sink. If $\alpha(i)\neq 0$, then 
$\alpha(i)=i$ since $\alpha$ is a path function. From $\alpha(i)=i$ and the definitions 
it follows that each maximal chain $X\in \mathfrak{M}_Q$ containing  $i$ belongs to 
$\mathrm{supp}(\alpha)$.  Therefore $\deg_{\mathrm{supp}(\alpha)}(i)=\deg_{Q}(i)$.
If $\alpha(i)=0$, then $\deg_{\mathrm{supp}(\alpha)}(i)=1$ as $\alpha$ is special.
 
Assume, finally, that $i\in\mathbf{K}(Q)$ is a source. If $\alpha(i)=i$, 
then any $X\in \mathfrak{M}_Q$ containing $i$ belongs to $\mathrm{supp}(\alpha)$ by construction
and hence $\deg_{\mathrm{supp}(\alpha)}(i)=\deg_{Q}(i)$. If $\alpha(i)=0$, then $\alpha(j)=0$
for all $j\in\overline{i}$ since $\alpha$ is monotone. Therefore none of the maximal chains starting
at $i$ belongs to $\mathrm{supp}(\alpha)$ and thus $i\not\in \mathrm{supp}(\alpha)$, which contradicts
our assumptions.

It is left to consider the case $\alpha(i)\not\in \{0,i\}$. Let $\alpha(i)=j$ and $X\in \mathfrak{M}_Q$
be the unique maximal chain containing $i$ and $j$. Let $Y\in \mathfrak{M}_Q$ be any other maximal
chain starting in $i$ and $s\in Y\setminus\{i\}$. Assume that $\alpha(s)\neq 0$. 
Then $\alpha(s)\in\overline{s}$ since
$\alpha$ is a path function. Further, $\overline{s}\cap X=\varnothing$ since $Q$ is a tree.
In particular, we have  $\overline{s}\cap \overline{\alpha(i)}=\varnothing$, which contradicts the assumption 
that $\alpha$ is monotone. Therefore $\alpha(s)=0$ and hence all non-leaves in $Y$ and all
edges in $Y$ are not in $\mathrm{supp}(\alpha)$ by construction.
Since $\mathrm{supp}(\alpha)$ is connected by assumptions and $Q$ is a tree,
it follows that $Y\cap \mathrm{supp}(\alpha)=i$.
Therefore $\deg_{\mathrm{supp}(\alpha)}(i)=1$ in this case.
The claim of the lemma follows.
\end{proof}


\begin{example}\label{nnewex6}
{\rm
If $Q$ is the quiver given by the left hand side of \eqref{eq101}, then possible supports
for special functions for $Q$ are: $Q_0$, $\{1,2\}$, $\{2,3\}$, $\{2,4\}$ and $\varnothing$.
The following is a complete list of special functions for $Q$ with support $Q_0$:
\begin{gather*}
(1,2,3,4),\,\,(2,2,3,4),\,\,(1,2,2,4),\,\,(1,2,3,2),\,\,(2,2,2,4),\\(1,2,2,2),\,\,(2,2,3,2),\,\,
(2,2,2,2).
\end{gather*}
For the same $Q$, here is a complete  list of special functions for $Q$ with support $\{1,2\}$:
\begin{displaymath}
(1,0,0,0),\,\,(2,0,0,0). 
\end{displaymath}
Finally, $(0,0,0,0)$ is the only special function for $Q$ with support $\varnothing$.
The total number of special functions in this case is $15$.
}
\end{example}

\begin{example}\label{nnewex6-1}
{\rm
If $Q$ is the quiver 
$\xymatrix{1\ar[r]&2&3\ar[l]}$, then possible supports
for special functions for $Q$ are: $Q_0$, $\{1,2\}$, $\{2,3\}$ and $\varnothing$.
The following is a complete list of special functions for $Q$ with support $Q_0$:
\begin{displaymath}
(1,2,3),\,\,(1,2,2),\,\,(2,2,3),\,\,(2,2,2).
\end{displaymath}
For the same $Q$, here is a complete  list of special functions for $Q$ with support $\{1,2\}$:
\begin{displaymath}
(1,0,0),\,\,(2,0,0). 
\end{displaymath}
Finally, $(0,0,0)$ is the only special function for $Q$ with support $\varnothing$.
The total number of special functions in this case is $9$.
}
\end{example}

\begin{example}\label{nnewex6-2}
{\rm
If $Q$ is the quiver 
$\xymatrix{1&2\ar[r]\ar[l]&3}$, then possible supports
for special functions for $Q$ are: $Q_0$, $\{1,2\}$, $\{2,3\}$ and $\varnothing$.
The following is a complete list of special functions for $Q$ with support $Q_0$:
\begin{displaymath}
(1,2,3),\,\,(1,2,0),\,\,(0,2,3),\,\,(0,2,0).
\end{displaymath}
For the same $Q$, here is a complete  list of special functions for $Q$ with support $\{1,2\}$:
\begin{displaymath}
(1,1,0),\,\,(0,1,0). 
\end{displaymath}
Finally, $(0,0,0)$ is the only special function for $Q$ with support $\varnothing$.
The total number of special functions in this case is $9$.
}
\end{example}

We denote by $\mathbf{C}$ the set of all special functions for $Q$.
A subtree $\Gamma$ of $Q$ (possibly empty) 
is called a {\em special subtree} if $\Gamma=\mathrm{supp}(\alpha)$ for some
special function $\alpha$. We denote by $\mathbf{W}$ the set of all special subtrees of $Q$.
We write
\begin{displaymath}
\mathbf{C}=\bigcup_{\Gamma\in \mathbf{W}} \mathbf{C}(\Gamma)
\end{displaymath}
where $\mathbf{C}(\Gamma)$ stands for the set of all special functions with support $\Gamma$.

\subsection{Type $A$ enumeration}\label{s7.25}

Here we enumerate special functions for type $A$ quivers.
In this subsection we let $Q$ be the oriented quiver obtained by choosing some orientation of 
the following Dynkin diagram of type $A_n$:
\begin{equation}\label{eq67}
\xymatrix{
1\ar@{-}[rr]&&2\ar@{-}[rr]&&\dots\ar@{-}[rr]&&n
}
\end{equation}
As mentioned above, we assume $n>1$. 
We write $\mathbf{K}(Q)=\{l_1,l_2,\dots,l_k\}$, where $1=l_1<l_2<\dots<l_k=n$.

\begin{lemma}\label{lem102}
Let $\alpha$ be a non-zero special function for $Q$. Then the set $\mathrm{supp}(\alpha)$ has the form
$\{l_i,l_i+1,\dots,l_j\}$ for some $i,j\in \{1,2,\dots,k\}$ with $i<j$.
\end{lemma}

\begin{proof}
By  definition, $\mathrm{supp}(\alpha)$ is connected, has more than one vertex and 
both leaves of $\mathrm{supp}(\alpha)$ belong to $\mathbf{K}(Q)$. The claim follows.
\end{proof}

After Lemma~\ref{lem102}, for $i,j\in \{1,2,\dots,k\}$ with $i<j$ we denote by 
$\mathbf{C}(i,j)$ the set of all special functions with support
$\{l_i,l_i+1,\dots,l_j\}$. For $m=0,1,2,\dots$, we denote by $\mathrm{cat}(m)$ 
the $m$-th Catalan number $\frac{1}{m+1}\binom{2m}{m}$ and set
$\underline{\mathrm{cat}}(m):=\mathrm{cat}(m)-1$. 
For $m=1,2,3,\dots$, we set 
\begin{displaymath}
\mathrm{cat}_1(m):=\mathrm{cat}(m)-\mathrm{cat}(m-1)\quad\text{ and }\quad
\underline{\mathrm{cat}}_1(m):=\mathrm{cat}_1(m)-1. 
\end{displaymath}
For $m=2,3,4,\dots$, we set 
\begin{displaymath}
\mathrm{cat}_2(m):=\mathrm{cat}(m)-2\mathrm{cat}(m-1)+\mathrm{cat}(m-2)\quad\text{ and }\quad
\underline{\mathrm{cat}}_2(m):=\mathrm{cat}_2(m)-1.
\end{displaymath}

\begin{proposition}\label{lem25}
{\tiny\hspace{2mm}}

\begin{enumerate}[$($i$)$]
\item\label{lem25.1} If $k=2$, then $|\mathbf{C}(1,2)|=\underline{\mathrm{cat}}(n+1)$.
\item\label{lem25.2} If $k>3$ and  $i\in\{2,3,\dots,k-2\}$, then 
$|\mathbf{C}(i,i+1)|=\underline{\mathrm{cat}}_2(l_{i+1}-l_i+2)$.
\item\label{lem25.3} If $k>4$ and  $i,j\in\{2,3,\dots,k-1\}$ with $j>i+1$, then 
\begin{displaymath}
|\mathbf{C}(i,j)|=\mathrm{cat}_1(l_{i+1}-l_i+1)\mathrm{cat}_1(l_{j}-l_{j-1}+1)
\prod_{s=i+1}^{j-2}\mathrm{cat}(l_{s+1}-l_s).
\end{displaymath}
\item\label{lem25.4} If $k>2$, then 
\begin{displaymath}
|\mathbf{C}(1,k)|=\mathrm{cat}(l_{2})\mathrm{cat}(l_k-l_{k-1}+1)
\prod_{s=2}^{k-2}\mathrm{cat}(l_{s+1}-l_s).
\end{displaymath}
\item\label{lem25.5} If $k>3$ and  $j\in\{3,4,\dots,k-1\}$, then 
\begin{displaymath}
|\mathbf{C}(1,j)|=\mathrm{cat}(l_2)\mathrm{cat}_1(l_{j}-l_{j-1}+1)
\prod_{s=2}^{j-2}\mathrm{cat}(l_{s+1}-l_s).
\end{displaymath}
\item\label{lem25.6} If $k>3$ and  $i\in\{2,3,\dots,k-2\}$, then 
\begin{displaymath}
|\mathbf{C}(i,k)|=\mathrm{cat}(l_k-l_{k-1}+1)\mathrm{cat}_1(l_{i+1}-l_{i}+1)
\prod_{s=i+1}^{k-2}\mathrm{cat}(l_{s+1}-l_s).
\end{displaymath}
\item\label{lem25.7} If $k>2$, then 
\begin{displaymath}
|\mathbf{C}(1,2)|=\underline{\mathrm{cat}}_1(l_2-l_1+2)\quad\text{ and }\quad
|\mathbf{C}(k-1,k)|=\underline{\mathrm{cat}}_1(l_k-l_{k-1}+2).
\end{displaymath}
\end{enumerate}
\end{proposition}

To prove this we will need the combinatorial lemmata below.
For $q\in\{1,2,\dots\}$, we set $\underline{q}:=\{1,2,\dots,q\}$.

\begin{lemma}\label{nlem51}
For $q\in\{1,2,\dots\}$, let $X_q$ denote the set of all transformations $f$ of 
$\underline{q}$ satisfying the conditions
\begin{enumerate}[$($i$)$]
\item\label{nlem51.1} $f(i)\leq i$ for all $i\in \underline{q}$ (i.e. $f$ is {\em order-decreasing});
\item\label{nlem51.2} $f(i)\leq f(j)$ for all $i,j\in \underline{q}$ such that $i\leq j$
(i.e. $f$ is {\em order-preserving}).
\end{enumerate}
Then $|X_q|=\mathrm{cat}(q)$.
\end{lemma}

\begin{proof}
See, for example, \cite{Hi}.
\end{proof}

\begin{lemma}\label{nlem52}
For $q\in\{2,3,\dots\}$, let $Y_q$ denote the set of all $f\in X_q$ such that 
$f(q)\neq q$.  Then $|Y_q|=\mathrm{cat}_1(q)$.
\end{lemma}

\begin{proof}
Let $Y'_q$ denote the set of all $f\in X_q$ such that $f(q)=q$. Then restriction to
$\underline{q-1}$ defines a bijection from $Y'_q$ to $X_{q-1}$.
Therefore the claim of our lemma follows from Lemma~\ref{nlem51}.
\end{proof}

\begin{lemma}\label{nlem53}
For $q\in\{3,4,\dots\}$, let $Z_q$ denote the set of all $f\in Y_q$ such that 
$f(2)=1$.  Then $|Z_q|=\mathrm{cat}_2(q)$.
\end{lemma}

\begin{proof}
Let $Z'_q$ denote the set of all $f\in Y_q$ such that $f(2)=2$. Then restriction to
$\{2,3,\dots,q\}$ followed by the identification of the latter set with
$\underline{q-1}$ given by $x\mapsto x-1$ gives rise to a bijection from $Z'_q$ to $Y_{q-1}$.
Therefore the claim of our lemma follows from Lemma~\ref{nlem52}.
\end{proof}

\begin{proof}[Proof of Proposition~\ref{lem25}.]
To prove claim~\eqref{lem25.1}, let $k=2$ and assume that $Q$ is given by  \eqref{eq2}.
For convenience, we write $n+1$ for $0$. For a function $\alpha:Q_0\to Q_0\cup\{n+1\}$,
let $\alpha'$ denote the extension of $\alpha$ to a transformation of $Q_0\cup\{n+1\}$
via $\alpha'(n+1)=n+1$. Then the fact that $\alpha$ is a path function is equivalent to
the requirement $\alpha'(i)\geq i$ for all $i$ (i.e. $\alpha$ is order increasing). 
Furthermore, the fact that $\alpha$ is monotone is equivalent to the requirement that $i\leq j$ 
implies $\alpha'(i)\leq \alpha'(j)$ for all $i,j$ (i.e. $\alpha$ is order preserving).
Thus the correspondence $\alpha\mapsto \alpha'$ defines a bijection between $\mathbf{C}(1,2)$
and the set of all order increasing and order preserving transformations of $Q_0\cup\{n+1\}$
which are different from the constant transformation with image $n+1$
(the latter is the unique constant order increasing and order preserving transformation). Hence 
$|\mathbf{C}(1,2)|=\mathrm{cat}(n+1)-1$, using Lemma~\ref{nlem51} with reversed order.

To prove claim~\eqref{lem25.2}, assume that $k>3$ and $i\in\{2,3,\dots,k-2\}$. Consider a
special function $\alpha$ supported on $\{l_i,l_i+1,\dots,l_{i+1}\}$,
in particular, $\alpha$ is zero outside this interval. Without loss of generality
we may assume that the arrows in this interval are of the form 
$x\to x-1$ (the other case is similar). Then, using Lemma~\ref{supportlemma}, we have  
$\alpha(l_i)=0$ and $\alpha(l_{i+1})<l_{i+1}$.
Similarly to the previous paragraph, the requirement that $\alpha$ is a path function is
equivalent to the fact that it decreases the order and the requirement that $\alpha$ is monotone
is equivalent to the fact that it preserves the order. Define a bijection from
$\{0,l_i,l_i+1,\dots,l_{i+1}\}$ to $\{1,2,\dots,q\}$, where $q=l_{i+1}-l_i+2$, by mapping 
$0$ to $1$ and $l_i+j$ to $2+j$, for $j=0,1,\dots$. Under this bijection, special
functions  supported on  $\{l_i,l_i+1,\dots,l_{i+1}\}$ are mapped exactly to those functions
in $Z_{q}$ which are different from the constant function $\underline{q}\to 1$.
Therefore claim~\eqref{lem25.2} follows from Lemma~\ref{nlem53}.

To prove claim~\eqref{lem25.7}, assume $k>2$. We prove the first equality, the second
one follows because of symmetry by swapping the order.  Consider a special function $\alpha$ supported on 
$\{1,2,\dots,l_2\}$, in particular, $\alpha$ is zero outside this interval. 
Without loss of generality we may assume that the arrows in this interval are of the form 
$x\to x-1$ (the other case is similar). Then $\alpha(l_{2})<l_{2}$.
Similarly to the above, there is a bijection between such functions and order decreasing
and order preserving transformations of $\underline{l_2}$ different from the 
unique constant transformation and satisfying $\alpha(l_{2})<l_{2}$. 
Therefore claim~\eqref{lem25.7} follows from Lemma~\ref{nlem52}.

To prove claim~\eqref{lem25.3}, assume $k>2$ and $i,j\in\{2,3,\dots,k-2\}$ are chosen such that $i<j$.
Consider a special function $\alpha$ supported on 
$\{l_i,l_i+1,\dots,l_j\}$, in particular, $\alpha$ is zero outside this interval.
Then $\alpha(l_i)\neq l_i$ and $\alpha(l_j)\neq l_j$. At the same time, 
$\alpha(l_s)=l_s$ for all $s$ such that $i<s<j$. The value of $\alpha$ can be chosen independently
on the intervals of the form $\{l_s,l_s+1,\dots,l_{s+1}\}$, where $s$ is such that $i\leq s<j$.
If $s\neq i,j-1$, then on this interval the values of $\alpha$ correspond precisely to 
order preserving and order decreasing (or increasing, depending on the orientation of arrows on
this interval) transformations of this interval, taking into account the conditions
$\alpha(l_s)=l_s$ and $\alpha(l_{s+1})=l_{s+1}$. Hence, from Lemma~\ref{nlem51} it follows that we have
$\mathrm{cat}(l_{s+1}-l_s)$ choices for the values of $\alpha$ on this interval. On the interval
$\{l_i,l_i+1,\dots,l_{i+1}\}$ we additionally have to take into account the condition 
$\alpha(l_i)\neq l_i$ to get exactly $\mathrm{cat}_1(l_{i+1}-l_i+1)$ choices by Lemma~\ref{nlem52}.
By symmetry, we have $\mathrm{cat}_1(l_{j}-l_{j-1}+1)$ choices for the last interval.
Now claim~\eqref{lem25.3} follows by applying the product rule.

Claims~\eqref{lem25.4}---\eqref{lem25.6} are proved similarly to claim~\eqref{lem25.3}. 
\end{proof}

\begin{example}\label{nnewex301}
{\rm  
For the quiver $\xymatrix{1\ar[r]&2&3\ar[l]&4\ar[l]\ar[r]&5&6\ar[l]}$, we have $k=5$, $l_2=2$, $l_3=4$
and $l_4=5$. Proposition~\ref{lem25}\eqref{lem25.3} says that there are exactly 
$\mathrm{cat}_1(3)\mathrm{cat}_1(2)=3$ special functions supported on $\{2,3,4,5\}$. These functions are
\begin{displaymath}
(0,0,3,4,0,0),\quad (0,0,2,4,0,0),\quad (0,0,0,4,0,0). 
\end{displaymath}
}
\end{example}

\subsection{Some notation for the path algebra}\label{s7.26}

Let us go back to an admissible tree quiver $Q$ as defined in Subsection~\ref{s7.2}. Let $A$ be the
path algebra of $Q$ over $\Bbbk$. For $i\in Q_0$ denote by $e_i$ the trivial path in vertex $i$. Then
$P_i=Ae_i$ and $L_i=P_i/\mathrm{Rad}(P_i)$. For each $i,j\in\{1,2,\dots,n\}$ such that 
$j\in\overline{i}$, denote by $\mathtt{a}_{ji}$ the unique path from $i$ to $j$.
Then $\{\mathtt{a}_{ji}\}$ is a basis in the one-dimensional vector space $e_jAe_i$.

From now on we assume that $\mathbf{K}'(Q)\neq\varnothing$. This is  equivalent to the requirement  that $Q$ 
is not isomorphic to the quiver given by \eqref{eq2}.

\subsection{Graph of the identity bimodule}\label{s7.3}

To study subbimodules of the identity bimodule, it is convenient to use a graphical presentation of the latter.
For this we consider ${}_AA_A$ as an $A\otimes A^{\mathrm{op}}$-module, cf. \cite[Chapter~II]{Ba}, or,
equivalently, as a representation of the quiver  $Q\times Q^{\mathrm{op}}$
where we impose all possible commutativity relations. We refer the reader to \cite{Sk} for details
concerning the isomorphism between $A\otimes A^{\mathrm{op}}$ and the quiver algebra,
see also \cite{ASS,Ri} for details on representations of quivers in general.

Viewing ${}_AA_A$ as a representation of $Q\times Q^{\mathrm{op}}$ (with relations) can be arranged into 
a graph, whose vertices are paths in $Q$, with left multiplication by arrows in $Q$ being depicted 
using solid arrow  and right multiplication by arrows in $Q$ being depicted by dashed arrows. 

\begin{example}\label{nnewex302}
{\rm  
If $Q$ is the quiver
\begin{equation}\label{eq6}
\xymatrix{ 
1\ar[r]&2&3\ar[l]&4\ar[l]\ar[r]&5\ar[r]&6,
}
\end{equation}
we obtain the following graphical presentation of ${}_AA_A$:
\begin{equation}\label{eq7}
\xymatrix{
\mathtt{a}_{11}\ar[r]&\mathtt{a}_{21}&&&&\\
&\mathtt{a}_{22}\ar@{-->}[u]\ar@{-->}[d]&&&&\\
&\mathtt{a}_{23}\ar@{-->}[d]&\mathtt{a}_{33}\ar[l]\ar@{-->}[d]&&&\\
&\mathtt{a}_{24}&\mathtt{a}_{34}\ar[l]&\mathtt{a}_{44}\ar[l]\ar[r]&\mathtt{a}_{54}\ar[r]&\mathtt{a}_{64}\\
&&&&\mathtt{a}_{55}\ar@{-->}[u]\ar[r]&\mathtt{a}_{65}\ar@{-->}[u]\\
&&&&&\mathtt{a}_{66}\ar@{-->}[u]\\
}
\end{equation}
Note that the rows of the above are in bijection with indecomposable projective left $A$-modules.
}
\end{example}

\begin{example}\label{nnewex3031}
{\rm  
If $Q$ is the quiver given by the left hand side of \eqref{eq101}, then we have the following 
graphical presentation of ${}_AA_A$:
\begin{equation}
\xymatrix{
&\mathtt{a}_{44}\ar[d]&\\
\mathtt{a}_{11}\ar[d]&\mathtt{a}_{24}&\mathtt{a}_{33}\ar[d]\\
\mathtt{a}_{21}&\mathtt{a}_{22}\ar@{-->}[u]\ar@{-->}[l]\ar@{-->}[r]&\mathtt{a}_{23}
}
\end{equation}
}
\end{example}

\subsection{Diagram of a subbimodule in ${}_AA_A$}\label{s7.4}

Viewing ${}_AA_A$ as a representation of $Q\times Q^{\mathrm{op}}$ with all commutativity relations,
as described in Subsection~\ref{s7.3}, subbimodules in ${}_AA_A$ are exactly subrepresentations.
As all composition multiplicities in $A$ are at most one (since there is at most one path between
any pair of vertices), it follows that subbimodules in $A$ are in bijection with those subsets of 
$\mathtt{a}_{ij}$'s which are closed under successors (i.e. under the action of both dashed and solid 
arrows).  In particular, the smallest subbimodules (i.e. the simple ones) are in bijection with the 
sinks in the diagram of $A$. Furthermore, we have the following easy observations:

\begin{lemma}\label{nlem72}
Let $B$ be a subbimodule of ${}_AA_A$. Then the set of 
all $\mathtt{a}_{ij}$ contained in $B$ is a basis of $B$.
\end{lemma}

\begin{proof}
This follows from  the proof of  Lemma~\ref{lem10}.
\end{proof}

\begin{lemma}\label{nlem71}
There is a bijection between $\mathfrak{M}_Q$ and simple subbimodules in the socle of ${}_A A_A$.
\end{lemma}

\begin{proof}
Each maximal chain with source $s$ and sink $t$ contributes the simple subbimodule in the
socle of ${}_A A_A$ with basis $\mathtt{a}_{ts}$.  Conversely, if $\mathtt{a}_{ts}$ belongs to the
socle, then it cannot be multiplied by any arrow from the left and, similarly, by any arrow from the
right. Therefore $\mathtt{a}_{ts}$ is a maximal path, that is it corresponds to a maximal chain.
\end{proof}

Lemma~\ref{nlem72} yields a graphic presentation of $B$ as
a subgraph of the graphic presentation of ${}_AA_A$ discussed in Subsection~\ref{s7.3}.

\begin{example}\label{nnewex303}
{\rm  
For the quiver \eqref{eq6} and the graph \eqref{eq7} of the corresponding identity bimodule, 
the graph of the subbimodule $J_4$, see \eqref{eqan1}, is given by:
\begin{displaymath}
\xymatrix{
\mathtt{a}_{11}\ar[r]&\mathtt{a}_{21}&&&&\\
&\mathtt{a}_{22}\ar@{-->}[u]\ar@{-->}[d]&&&&\\
&\mathtt{a}_{23}\ar@{-->}[d]&\mathtt{a}_{33}\ar[l]\ar@{-->}[d]&&&\\
&\mathtt{a}_{24}&\mathtt{a}_{34}\ar[l]&0\ar@{..>}[l]\ar@{..>}[r]&\mathtt{a}_{54}\ar[r]&\mathtt{a}_{64}\\
&&&&\mathtt{a}_{55}\ar@{-->}[u]\ar[r]&\mathtt{a}_{65}\ar@{-->}[u]\\
&&&&&\mathtt{a}_{66}\ar@{-->}[u]
}
\end{displaymath}
Here $0$ stands on the place of $\mathtt{a}_{44}$ which is missing in $J_4$ from the identity bimodule and 
the dotted arrows depict the  corresponding zero multiplication.
This clearly shows that $J_4$ is a decomposable bimodule. In particular, we obtain that in this case 
the monoid $\mathcal{I}$ should be  rather different from the multisemigroup $\mathcal{S}_{\ccP}$. 
Note that, for example, the linear span of $\mathtt{a}_{44}$ and $\mathtt{a}_{54}$ is not a subrepresentation
as it is not closed with respect to the action of the arrow $\mathtt{a}_{54}\to\mathtt{a}_{64}$.
Therefore this linear span is not a subbimodule. 
}
\end{example}

\subsection{Special function of an indecomposable subbimodule}\label{s7.5}

Let $B$ be a subbimodule of ${}_AA_A$. Then 
\begin{displaymath}
B=\bigoplus_{i,j=1}^n e_jBe_i 
\end{displaymath}
with each $e_jBe_i$ being of dimension at most one. Moreover, $e_jBe_i\neq 0$ implies that $j\in\overline{i}$.
For $i=1,2,\dots,n$, set 
\begin{displaymath}
B_i:=\bigoplus_{j=1}^n e_jBe_i.
\end{displaymath}
Note that $B_i$ is, by construction, 
a submodule of ${}_AB$ and also a submodule of $B\cap Ae_i$, where $Ae_i\cong P_i$. 

\begin{lemma}\label{nlem401}
Let $B$ be an indecomposable subbimodule of ${}_AA_A$. Then each $B_i$ is indecomposable or zero.
\end{lemma}

\begin{proof}
If $\deg_Q^{\mathrm{out}}(i)=1$, then $P_i$ is uniserial since $Q$ is admissible. Therefore
any submodule of $P_i$ is either indecomposable or zero. If $\deg_Q^{\mathrm{out}}(i)>1$,
then $i$ is a source since $Q$ is admissible. If $B_i\cong P_i$, then $B_i$ is indecomposable.

It remains to consider the case $\deg_Q^{\mathrm{out}}(i)>1$ and $B_i\not\cong P_i$.
Consider the full subgraph $Q^{(i)}$ of $Q$ with vertices $Q\setminus\{i\}$.
Let $\Gamma^{(j)}$, where $j=1,2,\dots,m$ for $m\geq 2$, be the list of all 
connected components of $Q^{(i)}$. For $j=1,2,\dots,m$, set
\begin{displaymath}
B_i^{(j)}:=\bigoplus_{t\in \Gamma^{(j)}} e_tBe_i
\quad
\text{ and we have }\quad
B_i=\bigoplus_{j=1}^m B_i^{(j)}.
\end{displaymath}
For each $j=1,2,\dots,m$, the space 
\begin{displaymath}
B_i^{(j)}\oplus \bigoplus_{t\in \Gamma^{(j)}}B_t 
\end{displaymath}
is a direct summand of $B$ as an $A$-$A$--bimodule. Since $B$ is assumed to be indecomposable,
we either have $B_i=0$ or $B_i=B_i^{(j)}$ for some $j$. Since $Q$ is admissible, 
$B_i^{(j)}$ is uniserial and hence indecomposable or zero.
\end{proof}

Lemma~\ref{nlem401} justifies the following definition.
For an indecomposable $B$ define the function $\mathbf{x}_B:Q_0\to Q_0\cup\{0\}$,
in the  following way:
\begin{itemize}
\item if $B_i=0$ for $i\in Q_0$, then set $\mathbf{x}_B(i)=0$;
\item if $B_i\neq 0$ for $i\in Q_0$, then $B_i$ is an indecomposable module
by Lemma~\ref{nlem401} and is projective since $A$ is hereditary, so we can
define $\mathbf{x}_B(i)$ as the unique element in 
$Q_0$ such that $B_i\cong P_{\mathbf{x}_B(i)}$.
\end{itemize}
We also define the {\em support} $\mathrm{supp}(B)$ of $B$ as the union of all maximal chains
in $Q$ which contribute to the bimodule socle of $B$, confer Lemma~\ref{nlem71}.

\begin{example}\label{nnewex311}
{\rm  
The bimodule given in Example~\ref{nnewex303} decomposes into a direct sum of two indecomposable 
summands. The first summand corresponds to the part on the left from $0$. This summand has
the function $(1,2,3,3,0,0)$, support $\{1,2,3,4\}$ and a two-dimensional socle with basis
$\mathtt{a}_{21}$ and $\mathtt{a}_{24}$. The second summand corresponds 
to the part on the right from $0$. This summand has
the function $(0,0,0,5,5,6)$, support $\{4,5,6\}$ and simple socle $\mathtt{a}_{64}$.
}
\end{example}

\begin{lemma}\label{nlem85}
The support of an indecomposable subbimodule $B$ in ${}_AA_A$ is connected. 
\end{lemma}

\begin{proof}
Assume $\mathrm{supp}(B)$ is the disjoint union of two non-empty sets $\Gamma_1$ and $\Gamma_2$.
Each of these is a  union of maximal chains in $Q$. Let $i\in Q_0$
be such that $B_i\neq 0$. Then $AB_iA$ intersects the socle of $B$ and hence
$i$ belongs to some maximal chain $X\subset \mathrm{supp}(B)$, in particular,
$i\in \Gamma_1$ or $i\in \Gamma_2$.

For $s=1,2$, let $B^{(s)}$ be the $\Bbbk$-span of all $B_j$, where $j\in \Gamma_s$.
Then $B=B^{(1)}\oplus B^{(2)}$. By construction, both $B^{(1)}$ and $B^{(2)}$ 
are left $A$-submodules of $B$. Since each $\Gamma_s$ is a union of maximal chains,
each $B^{(s)}$ is even a right $A$-submodule of $B$, in particular, a right 
$A$-$A$--subbimodule. Therefore $B$ is decomposable.
\end{proof}

\begin{proposition}\label{prop26}
Let $B$ be an indecomposable subbimodule of ${}_AA_A$. 
Then $\mathbf{x}_B$ is a special function and $\mathrm{supp}(\mathbf{x}_B)=\mathrm{supp}(B)$.
\end{proposition}

\begin{proof}
If $P_j$ is a submodule of $P_i$, then there is an oriented path from $i$ to $j$ in $Q$.
Therefore $\mathbf{x}_B(i)\in\overline{i}$ and thus $\mathbf{x}_B$ is a path function.

Assume that there is an oriented path $\alpha$ from $i$ to $j$ in $Q$
and that $\mathbf{x}_B(j)\neq 0$. Then right multiplication
with $\alpha$ defines an injective homomorphism from $P_j$ to $P_i$ inside $A$. Restricting this homomorphism
to $B$ gives an injective homomorphism from $B_j\cong P_{\mathbf{x}_B(j)}$
to $B_i\cong P_{\mathbf{x}_B(i)}$. This means that $\mathbf{x}_B(i)\neq 0$
and $\mathbf{x}_B(j)\in\overline{\mathbf{x}_B(i)}$. Therefore $\mathbf{x}_B$ is a monotone function.

Let $X\in\mathfrak{M}_Q$ with source $i$ and sink $j$. If $\mathbf{x}_B(i)\not\in X$, then 
$\mathbf{x}_B(s)=0$ for all $s\in X\setminus\{i\}$ since  $\mathbf{x}_B$ is a monotone path function.
In particular, $X$ does not contribute to the support of $\mathbf{x}_B$ (note that either $i$
or $j$ might still belong to the support via some other maximal chains). In this case we
also have $\mathtt{a}_{ji}\not\in B$.

Let $X\in\mathfrak{M}_Q$ with source $i$ and sink $j$. 
If $\mathbf{x}_B(i)\in X$, then $X\subset\mathrm{supp}(\mathbf{x}_B)$ and $\mathtt{a}_{ji}$
spans a simple subbimodule in the socle of $B$. From this and the previous paragraph
it follows that $\mathrm{supp}(\mathbf{x}_B)=\mathrm{supp}(B)$.
In particular, from Lemma~\ref{nlem85} we obtain that $\mathrm{supp}(\mathbf{x}_B)$ is connected.

It remains to show that the equality  $\mathbf{x}_B(i)=0$ for 
$i\in \mathbf{K}(Q)\cap \mathrm{supp}(B)$ implies that $\deg_{\mathrm{supp}(B)}(i)=1$.
If $i$ were a source, then $\mathbf{x}_B(i)=0$ would imply $i\not\in \mathrm{supp}(B)$, which is a contradiction.
Therefore $i$ is a sink. Consider the full subgraph $Q^{(i)}$ of $Q$ with vertices
$Q\setminus\{i\}$. Let $\Gamma^{(j)}$, for $j=1,2,\dots,m$ where $m\geq 1$, be the list of all 
connected components of $Q^{(i)}$. For $j=1,2,\dots,m$, denote by $B^{(j)}$ the  $A$-$A$-bimodule
direct summand of $B$ spanned, over $\Bbbk$, by all $B_s$, where $s\in \Gamma^{(j)}$.
Since $B$ is indecomposable, only one of these direct summands is non-zero. Without loss of
generality we assume that this non-zero direct summand is $B^{(1)}$. There are $m$ maximal chains with
sink $i$, one for each $\Gamma^{(j)}$. Since only $B^{(1)}$ is non-zero, only the maximal chain from
$\Gamma^{(1)}$ contributes to the socle of $B$. Hence  $\deg_{\mathrm{supp}(B)}(i)=1$.
This completes the proof.
\end{proof}

\subsection{Subbimodules of ${}_AA_A$ associated with special functions}\label{s7.7}

For a special function  $\mathbf{x}=(x_1,x_2,\dots,x_n):Q_0\to Q_0\cup\{0\}$, denote by  $B_{\mathbf{x}}$ the 
subspace in ${}_AA_A$ obtained as the linear span of all $\mathtt{a}_{ts}$ for which 
$x_s\neq 0$ and $t\in \overline{x_s}$. The fact that $\mathbf{x}$ is a path 
function ensures that this definition does make sense. Moreover, for $s\in Q_0$ we have
\begin{equation}\label{nneq17}
(B_{\mathbf{x}})_s\cong 
\begin{cases}
P_{x_s}, & x_s\neq 0,\\
0, & \text{ otherwise},
\end{cases}
\end{equation}
by construction.

\begin{example}\label{ex26}
{\rm
For the quiver
\begin{displaymath}
\xymatrix{
1\ar[r]&2\ar[r]&3&4\ar[l]&5\ar[r]\ar[l]&6\ar[r]&7\ar[r]&8&9\ar[l]
}
\end{displaymath}
and the special function given by  $(0,0,0,3,5,7,8,8,8)$, we have the subbimodule of 
${}_AA_A$ given by the bold and solid part of the diagram in Figure~\ref{fig1}, where the 
the rest of ${}_AA_A$ is shown as connected by dotted arrows.
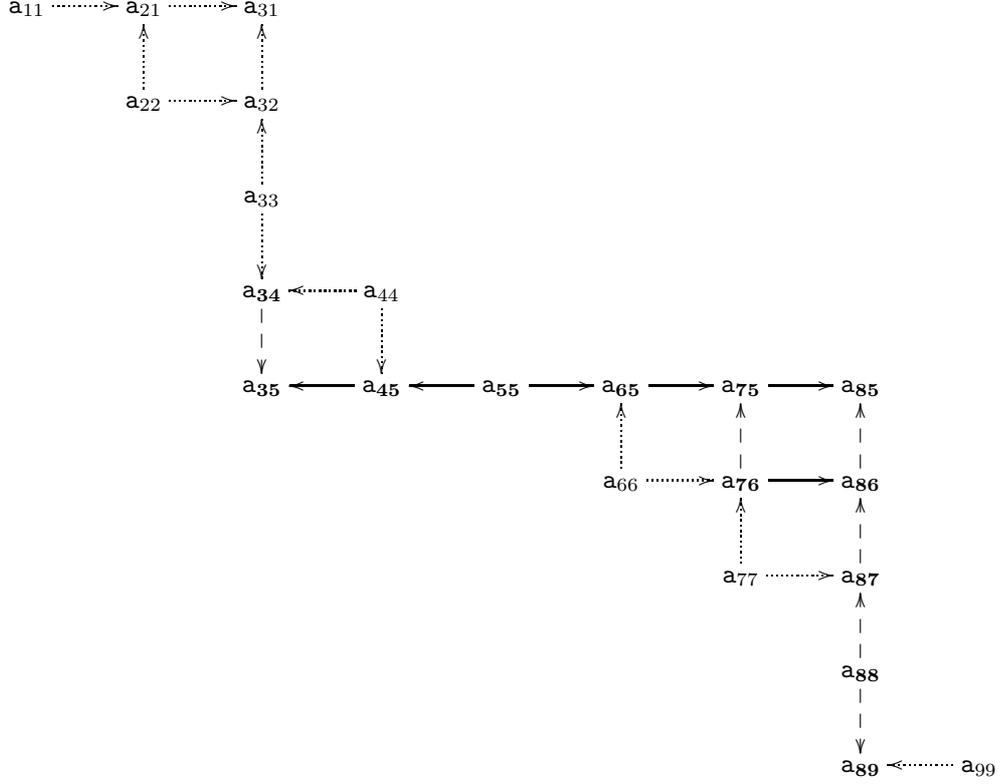
\begin{figure}
\begin{displaymath}
\xymatrix{
\mathtt{a}_{11}\ar@{.>}[r]&\mathtt{a}_{21}\ar@{.>}[r]&\mathtt{a}_{31}&&&&&&\\
&\mathtt{a}_{22}\ar@{.>}[r]\ar@{.>}[u]&\mathtt{a}_{32}\ar@{.>}[u]&&&&&&\\
&&\mathtt{a}_{33}\ar@{.>}[u]\ar@{.>}[d]&&&&&&\\
&&\boldsymbol{\mathtt{a}_{34}}\ar@{-->}[d]&\mathtt{a}_{44}\ar@{.>}[d]\ar@{.>}[l]&&&&&\\
&&\boldsymbol{\mathtt{a}_{35}}&\boldsymbol{\mathtt{a}_{45}}\ar[l]&\boldsymbol{\mathtt{a}_{55}}\ar[l]\ar[r]&
\boldsymbol{\mathtt{a}_{65}}\ar[r]&\boldsymbol{\mathtt{a}_{75}}\ar[r]&\boldsymbol{\mathtt{a}_{85}}&\\
&&&&&\mathtt{a}_{66}\ar@{.>}[r]\ar@{.>}[u]&\boldsymbol{\mathtt{a}_{76}}\ar[r]\ar@{-->}[u]&
\boldsymbol{\mathtt{a}_{86}}\ar@{-->}[u]&\\
&&&&&&\mathtt{a}_{77}\ar@{.>}[r]\ar@{.>}[u]&\boldsymbol{\mathtt{a}_{87}}\ar@{-->}[u]&\\
&&&&&&&\boldsymbol{\mathtt{a}_{88}}\ar@{-->}[u]\ar@{-->}[d]&\\
&&&&&&&\boldsymbol{\mathtt{a}_{89}}&\mathtt{a}_{99}\ar@{.>}[l]\\
}
\end{displaymath}
\caption{Diagram used in Example~\ref{ex26}.}\label{fig1}
\end{figure}
}
\end{example}

\begin{proposition}\label{prop29}
For every special function $\mathbf{x}$, the subspace $B_{\mathbf{x}}$ of ${}_AA_A$ is an indecomposable 
subbimodule and $\mathrm{supp}(\mathbf{x})=\mathrm{supp}(B_{\mathbf{x}})$.
\end{proposition}

\begin{proof}
From \eqref{nneq17} it follows that $B_{\mathbf{x}}$ is closed with respect to the left $A$-action.
From the fact that $\mathbf{x}$ is monotone, it follows that $B_{\mathbf{x}}$ is closed with 
respect to the right $A$-action. Therefore $B_{\mathbf{x}}$ is a subbimodule of ${}_AA_A$.

Let $X\in\mathfrak{M}_Q$ be a maximal chain with source $i$ and sink $j$. If $\mathbf{x}(i)\not\in X$, then 
$X\not\subset \mathrm{supp}(\mathbf{x})$ by definition and $X\not\subset \mathrm{supp}(B_{\mathbf{x}})$
by construction (as $\mathtt{a}_{ji}\not\in B_{\mathbf{x}}$). If $\mathbf{x}(i)\in X$, then 
$X\subset \mathrm{supp}(\mathbf{x})$ by definition and $X\not\subset \mathrm{supp}(B_{\mathbf{x}})$
by construction (as $\mathtt{a}_{ji}\in B_{\mathbf{x}}$). Hence 
$\mathrm{supp}(\mathbf{x})=\mathrm{supp}(B_{\mathbf{x}})$.

It remains to show that $B_{\mathbf{x}}$ is indecomposable. Assume that this is not the case and write
$B_{\mathbf{x}}\cong B^{(1)}\oplus B^{(2)}$, where both direct summands are subbimodules. Since 
$\mathrm{supp}(B_{\mathbf{x}})=\mathrm{supp}(\mathbf{x})$ is connected, $\mathrm{supp}(B^{(1)})$
and $\mathrm{supp}(B^{(2)})$ must have at least one common element, say $s$. This common element 
can be chosen such that it is
an end element of a maximal chain, hence it is either a source of a sink. Assume first that $s$
is a source. Then $\mathbf{x}(s)=s$ implies that $\mathtt{a}_{ss}$ must be in both $B^{(1)}$
and $B^{(2)}$ since $(B_{\mathbf{x}})_s\cong P_s$ is indecomposable, which contradicts the fact that
the intersection of these two subbimodules is trivial. Further, $\mathbf{x}(s)\neq s$ 
implies that only one maximal chain 
starting with $s$, namely the one containing $\mathbf{x}(s)$, can contribute to the socle of 
$B_{\mathbf{x}}$. That socle element must be either in $B^{(1)}$ or $B^{(2)}$, but it cannot be 
in both of them, which contradicts the fact that $s$ is a common element in 
$\mathrm{supp}(B^{(1)})$ and $\mathrm{supp}(B^{(2)})$. 

Therefore $s$ is a sink. If $\mathbf{x}(s)=s$, then for any $X\in \mathrm{supp}(B_{\mathbf{x}})$
with sink $s$ one can use right multiplication with elements in $A$ to move $\mathtt{a}_{ss}$ to the 
socle element of $B_{\mathbf{x}}$ corresponding to $X$. Therefore $\mathtt{a}_{ss}$ must be 
in both $B^{(1)}$ and $B^{(2)}$, which gives a similar contradiction 
to the above. If $\mathbf{x}(s)\neq s$, then $\mathbf{x}(s)=0$
since $s$ is a sink. By definition of a special function, we thus have that $s$ has degree $1$ in
$\mathrm{supp}(\mathbf{x})$. Therefore only one maximal chain  ending at $s$ can contribute to the socle of 
$B_{\mathbf{x}}$. That socle element must be either in $B^{(1)}$ or $B^{(2)}$, but it cannot be 
in both of them, which is again a similar contradiction to the above. This completes the proof.
\end{proof}

\subsection{Classification of indecomposable subbimodules of ${}_AA_A$}\label{s7.8}
We can now collect the above facts into the following statement.

\begin{theorem}\label{thm31}
The maps  $B\mapsto  \mathbf{x}_B$ and $\mathbf{x}\mapsto B_{\mathbf{x}}$ are mutually inverse bijections
between the set of all indecomposable subbimodules of ${}_AA_A$ and the set of all special functions.
\end{theorem}

\begin{proof}
Let $\mathbf{x}$ be a special function. Then $B_{\mathbf{x}}$ is an indecomposable subbimodule of 
${}_AA_A$ by Proposition~\ref{prop29}. Furthermore, 
$\mathbf{x}_{B_{\mathbf{x}}}(i)=\mathbf{x}(i)$ for all $i\in Q_0$, and thus also
$\mathbf{x}_{B_{\mathbf{x}}}=\mathbf{x}$ follow from the definitions. 

Conversely, let $B$ is an indecomposable subbimodule of ${}_AA_A$. Then $\mathbf{x}_B$ is a special
function by Proposition~\ref{prop26}. Furthermore, 
$(B_{\mathbf{x}_{B}})_i=B_i$ for all $i\in Q_0$, and thus also
$B_{\mathbf{x}_{B}}=B$ follow from the definitions. 
\end{proof}

For a non-empty $\Gamma\in\mathbf{W}$, we denote by $\mathbf{B}(\Gamma)$ the set of all indecomposable subbimodules 
$B\subset {}_AA_A$ for which $\mathrm{supp}(B)=\Gamma$. We set $\mathbf{B}(\varnothing)=\{0\}$.

\subsection{Partial order}\label{s7.9}

We identify $\mathcal{S}_{\ccP}$ with the subset $\mathcal{I}^{\mathrm{ind}}$ of $\mathcal{I}$ consisting of
all indecomposable subbimodules to which we attach an external element $0$ (which corresponds to the zero bimodule). 
To this end, we do not know whether $\mathcal{I}^{\mathrm{ind}}$ is a 
submonoid of $\mathcal{I}$, that is, whether $IJ\in \mathcal{I}^{\mathrm{ind}}$ for any $I,J\in
\mathcal{I}^{\mathrm{ind}}$. The original multivalued operation in $\mathcal{I}^{\mathrm{ind}}$ sends
$(I,J)$ to the set of all indecomposable direct summands of $IJ$ up to isomorphism.

The set $\mathcal{I}^{\mathrm{ind}}$ inherits from $\mathcal{I}$ the partial order given by inclusions. 
Clearly, ${}_AA_A$ is the maximum element with respect to this order both in 
$\mathcal{I}^{\mathrm{ind}}$ and in $\mathcal{I}$ (note that ${}_AA_A\in \mathcal{I}^{\mathrm{ind}}$
as we assume $Q$ to be connected).

Consider the set $\mathbf{Q}:=\{J_s\,:\,s=1,2,\dots,n\}$ and note that this is the set of maximal elements 
in $\mathcal{I}\setminus\{{}_AA_A\}$. The bimodule $J_s$ is indecomposable 
if and only if we have $s\not\in\mathbf{K}'(Q)$. For $s\in \mathbf{K}'(Q)$, let $t_1,t_2,\dots,t_{m_s}$
be the list of all $t\in Q_0$ for which  there is an arrow $t\to s$ or an arrow $s\to t$. Let 
$\Gamma$ be the full subgraph of $Q$ with vertex set $Q_0\setminus\{s\}$. Then 
\begin{displaymath}
\Gamma=\Gamma^{(1)}\cup \Gamma^{(2)}\cup\dots \cup \Gamma^{(m_s)}  
\end{displaymath}
where $\Gamma^{(q)}$ is the connected component containing $t_q$ for $q=1,2,\dots,m_s$. We have the decomposition 
\begin{equation}\label{eq452}
J_s\cong \bigoplus_{q=1}^{m_s} J_s^{(q)} 
\end{equation}
where $J_s^{(q)}$ is the subbimodule of $J_s$ defined as the direct sum of all $e_j(J_s)e_i$ 
with $i,j\in \Gamma^{(q)}\cup\{s\}$. Clearly, each $J_s^{(q)}$ is indecomposable
since $\Gamma^{(q)}$ is connected.

\begin{lemma}\label{lem41}
{\hspace{2mm}}

\begin{enumerate}[$($i$)$]
\item\label{lem41.1} The set $\{J_s\,:\,s\not\in\mathbf{K}'(Q)\}$ is the set of maximal elements in 
$\mathbf{B}(Q)\setminus\{{}_AA_A\}$.
\item\label{lem41.2} For $\Omega\in\mathbf{W}\setminus \{Q\}$, there is a unique maximal element,
denoted $B_{\Omega}$, in the set $\mathbf{B}(\Omega)$. Moreover,  $B_{\varnothing}=0$ and, for 
$\Omega\neq \varnothing$, we have
\begin{equation}\label{eqeq12}
B_{\Omega}=\prod_{t} J_t^{(p_t)}\prod_{s} J_s^{(q_s)}
\end{equation}
where $s$ runs through the set of sinks $i\in \mathbf{K}'(Q)\cap \Omega$ for which 
$\deg_{\Omega}(i)=1$ and $q_s$ is such that the corresponding $\Gamma^{(q_s)}$ has a common vertex with $\Omega$,
while $t$ runs through the set of sources $i\in \mathbf{K}'(Q)\cap \Omega$ for which 
$\deg_{\Omega}(i)=1$ and $p_t$ is such that the corresponding $\Gamma^{(p_t)}$ has a common vertex with $\Omega$.
\end{enumerate}
\end{lemma}

\begin{proof}
Clearly each $J_s$ with $s\not\in\mathbf{K}'(Q)$ is maximal in $\mathbf{B}(Q)\setminus\{{}_AA_A\}$.
Assume that $B\in \mathbf{B}(Q)$ is maximal in $\mathbf{B}(Q)\setminus\{{}_AA_A\}$. Then 
$B\subset J_s$ for some $s$. If $s\not\in\mathbf{K}'(Q)$, then $B=J_s$ by maximality of $B$.
If $s\in\mathbf{K}'(Q)$, then 
\begin{displaymath}
B=\bigoplus_{q=1}^{m_s}(B\cap J_s^{(q)}) 
\end{displaymath}
since $B$ has a basis consisting of all $\mathtt{a}_{st}$ contained in it and
all $J_s^{(q)}$ also have the same property. By indecomposability, we get 
$B=B\cap J_s^{(q)}$ for some $q$, which contradicts $B\in \mathbf{B}(Q)$. 
Therefore this case does not occur, which proves  claim~\eqref{lem41.1}.

To prove claim~\eqref{lem41.2} we denote by  $B'_{\Omega}$ the right hand side of \eqref{eqeq12}.
Note that $B_{\varnothing}=0$ is clear and that for $\Omega\neq\varnothing$ the fact that 
$B'_{\Omega}\in \mathbf{B}(\Omega)$ follows by construction. The maximal element
$B_{\Omega}$ in the set $\mathbf{B}(\Omega)$ is the sum of all subbimodules of ${}_A A_A$ with support $\Omega$.
Therefore to complete the proof of claim~\eqref{lem41.2} it remains to check that $B'_{\Omega}=B_{\Omega}$
for $\Omega\neq\varnothing$.

If there are $s$ and $t$ in \eqref{eqeq12} which are connected by an edge, then $\Omega$ must be the full
subgraph of $Q$ with vertices $\{s,t\}$ by connectedness. In this case the only subbimodule of 
${}_A A_A$ with support $\{s,t\}$ is the one with basis $\mathtt{a}_{st}$. Indeed, since both $s$ and 
$t$ have degrees higher than $1$, appearance of either $\mathtt{a}_{ss}$ or $\mathtt{a}_{tt}$
in a subbimodule would lead to an extra maximal chain in its support. Therefore $B_{\Omega}$ has basis
$\mathtt{a}_{st}$. At the same time, $\mathtt{a}_{st}$ appears in $B'_{\Omega}$ as the 
product $\mathtt{a}_{st}\mathtt{a}_{tt}$, where
$\mathtt{a}_{st}\in J_t^{(p_t)}$ and $\mathtt{a}_{tt}\in J_s^{(q_s)}$. 
This means that $B'_{\Omega}=B_{\Omega}$. 

In the remaining case (no  $s$ and $t$ in 
\eqref{eqeq12} are connected by an edge), all factors of \eqref{eqeq12} commute. Note that 
$B_{\Omega}\subset J_s$ and $B_{\Omega}\subset J_t$ for any $s$ and $t$ occurring in  \eqref{eqeq12}.
From indecomposability, it follows that $B_{\Omega}\subset J_s^{(q_s)}$ and $B_{\Omega}\subset J_t^{(p_t)}$ for 
all $s$ and $t$ occurring in  \eqref{eqeq12}. From $J_s^2=J_s$ and $J_t^2=J_t$ it follows that
$(J_s^{(q_s)})^2=J_s^{(q_s)}$ and $(J_t^{(p_t)})^2=J_t^{(p_t)}$. This implies 
$B_{\Omega}J_s^{(q_s)}=B_{\Omega}$ and $B_{\Omega}J_t^{(p_t)}=B_{\Omega}$ which yields
$B_{\Omega}B'_{\Omega}=B_{\Omega}\subset B'_{\Omega}$. From the maximality of $B_{\Omega}$ we finally obtain that $B_{\Omega}=B'_{\Omega}$.
\end{proof}

\subsection{Composition of indecomposable subbimodules}\label{s7.10}
The following is a crucial observation.

\begin{proposition}\label{prop32}
Let $B$ and $D$ be two indecomposable subbimodules in ${}_AA_A$. Then $B\otimes_AD\cong BD$
and the latter is either zero or an indecomposable subbimodule of ${}_AA_A$.
\end{proposition}

\begin{proof}
As $A$ is hereditary, $B\otimes_A{}_-$ is exact, in particular, it preserves inclusions. Hence,
applying it to $D\hookrightarrow A$ gives $B\otimes_A D\hookrightarrow B\otimes_AA\cong B$ where
the last isomorphism is given by the multiplication map. Therefore the multiplication map 
$B\otimes_AD\to BD$ is an isomorphism.

It remains to prove indecomposability of $BD$ in case the latter subbimodule is nonzero. 
Let $\Omega:=\mathrm{supp}(B)\cap\mathrm{supp}(D)$ which is connected as both 
$\mathrm{supp}(B)$ and $\mathrm{supp}(D)$ are. As $BD\subset B\cap D$, we have 
$\mathrm{supp}(BD)\subset \Omega$. If $\Omega$ consists of only one maximal chain,
it follows that $\mathrm{supp}(BD)=\Omega$, that is $BD$ has simple socle and thus is indecomposable.

Let $X$ and $Y$ be two different maximal chains in $\Omega$ with a common vertex $i$. Then $i$ 
is either a sink or a source of degree at least two. Note that both $X$ and $Y$ belong to both 
$\mathrm{supp}(B)$ and  $\mathrm{supp}(D)$. As the degree of $i$ is at least two and $B$
is indecomposable, we have $\mathtt{a}_{ii}\in B$ (for otherwise we may decompose $B$ using the
decomposition of $J_i$). Similarly, $\mathtt{a}_{ii}\in D$. Hence $\mathtt{a}_{ii}\in BD$ as well.
Using left multiplication, in case $i$ is a source, or right multiplication, in case $i$ is a sink,
it follows that $\mathrm{supp}(BD)$ contains both $X$ and $Y$. Consequently,
$\mathrm{supp}(BD)=\Omega$ as the latter one is connected.

Assume that we can write $BD=B^{(1)}\oplus B^{(2)}$, where $B^{(1)}$ and  $B^{(2)}$ are non-zero
subbimodules. As $\Omega$ is connected, there must be some vertex, say $i$, in the intersection of 
the supports of these two subbimodules. From the previous paragraph we have $\mathtt{a}_{ii}\in BD$.
Now, using arguments similar to the one used in the proof of Proposition~\ref{prop29},
we get a contradiction. This completes the proof.
\end{proof}

An immediate consequence of Proposition~\ref{prop32} is the following:

\begin{corollary}\label{cor33}
The multisemigroup $\mathcal{S}_{\ccP}$ is a monoid. 
\end{corollary}

Another consequence from the proof of Proposition~\ref{prop32} is the following:

\begin{corollary}\label{cor33-1}
Let $B$ and $D$ be two indecomposable subbimodules in ${}_AA_A$ such that  $BD\neq 0$.
Then $\mathrm{supp}(BD)=\mathrm{supp}(B)\cap\mathrm{supp}(D)$.
\end{corollary}

We denote by $\mathcal{I}^{\mathrm{ind}}$ the submonoid of $\mathcal{I}$ consisting of indecomposable
subbimodules in ${}_AA_A$ and the zero bimodule. By the above, the 
monoids $\mathcal{S}_{\ccP}$ and $\mathcal{I}^{\mathrm{ind}}$ are isomorphic.

\begin{problem}\label{quest123}
{\rm
It would be interesting to know for which finite dimensional algebras the product of two indecomposable
subbimodules of the identity bimodule is always indecomposable or zero.
}
\end{problem}

\section{Presentation for $\mathcal{I}$ and $\mathcal{I}^{\mathrm{ind}}$}\label{s8}

The main aim of this section is to obtain presentations for both the monoid $\mathcal{I}$ and the monoid $\mathcal{I}^{\mathrm{ind}}$.

\subsection{Minimal generating systems}\label{s8.1}

Set
\begin{displaymath}
\mathbf{B}:=\{J_s\,:\,s\not\in\mathbf{K'}(Q)\}\cup
\bigcup_{s\in \mathbf{K'}(Q)}\{J_s^{(q)}\,:\,q=1,2,\dots,m_s\}.
\end{displaymath}
We will need the following technical observation.

\begin{lemma}\label{nnlem502}
For all $i,j\in Q_0$, we have 
\begin{equation}\label{eq71}
J_iP_j=
\begin{cases}
P_j,& i\neq j;\\
\mathrm{Rad}(P_i), & i=j.
\end{cases}
\end{equation}
Moreover, we also have $J_i\mathrm{Rad}(P_i)=\mathrm{Rad}(P_i)$.
\end{lemma}

\begin{proof}
If $i\neq j$, then $J_i$ contains $\mathtt{a}_{jj}$ by definition and thus $J_iP_j=P_j$.
If $i=j$, then $J_i$ does not contain $\mathtt{a}_{ii}$ by definition and thus $J_iP_i\neq P_i$.
However, since $A$ is hereditary, $\mathrm{Rad}(P_i)$ is a direct sum of some $P_k$, where
$k\neq i$. Therefore from the first sentence of the proof we have both $J_iP_i=\mathrm{Rad}(P_i)$
and $J_i\mathrm{Rad}(P_i)=\mathrm{Rad}(P_i)$.
\end{proof}

\begin{proposition}\label{prop41}
{\hspace{2mm}}

\begin{enumerate}[$($i$)$]
\item\label{prop41.1} The set $\mathbf{Q}$ is the unique minimal 
generating system for the monoid $\mathcal{I}$.
\item\label{prop41.2} 
The set $\mathbf{B}$ is the unique minimal generating system for the monoid $\mathcal{I}^{\mathrm{ind}}$.
\end{enumerate}
\end{proposition}

\begin{proof}[Proof of claim~\eqref{prop41.1}.]
As $\mathbf{Q}$ is the set of all maximal elements in $\mathcal{I}\setminus\{{}_AA_A\}$, it must belong
to any generating system. Therefore, to prove claim~\eqref{prop41.1} it is enough to show that 
$\mathbf{Q}$ generates $\mathcal{I}$. Let $S$ be the submonoid of $\mathcal{I}$ generated by $\mathbf{Q}$.
Assume that $\mathcal{I}\setminus S\neq\varnothing$ and let $B$ be a maximal element in 
$\mathcal{I}\setminus S$ with respect to inclusions. Certainly, $B\neq 0$, $B\neq {}_AA_A$
and $B\neq J_s$ for $s\in Q_0$.

We split the proof of claim~\eqref{prop41.1} into two cases.

{\bf Case~1.} Assume first that $B_i\in\{0,P_i\}$ for all $i$. 
As $B\neq {}_AA_A$, there is at least one $i$ such that
$B_i=0$. As $B\neq 0$, there is at least one $j$ such that $B_j=P_j$.  As $B$ is a bimodule, $B_i=0$ 
implies $B_s=0$ for all $s\in \overline{i}$. Therefore we may choose $i$ such that $B_i=0$
and for any arrow $j\to i$, where  $j\in Q_0$, we have $B_j=P_j$. We have to consider two subcases.

{\bf Subcase~1.A.} Assume that $i$ is not a source. Then
there is some $j$ with an arrow $j\to i$. In
particular, $P_i$ has simple socle, say $L_s$. Then $s$ is a sink and $B_s=0$. 

\begin{lemma}\label{nnlem501}
With the notation above, the space  $B':=B\oplus\Bbbk\{\mathtt{a}_{si}\}$ is a subbimodule of ${}_AA_A$.
\end{lemma}

\begin{proof}
By construction, $\Bbbk\{\mathtt{a}_{si}\}$ is the socle of $P_i$. Therefore $B'$ is a left module.
For any $j'\in Q_0$ such that there is an arrow $\alpha:j'\to i$, we have $B_{j'}=P_{j'}$ by our choice
of $i$. Therefore, right multiplication with $\alpha$ maps $\mathtt{a}_{si}$ to $\mathtt{a}_{sj'}\in B$.
The claim follows.
\end{proof}

Using \eqref{eq71} and $B_s=0$, we get $B=J_sB'$. As $B\subsetneq B'$, we
have $B'\in S$ by maximality of $B$. Therefore $B\in S$, a contradiction.

{\bf Subcase~1.B.} Assume that $i$ is a source and let $X$ be a maximal chain starting at $i$
and ending at some $s\in Q_0$. Then $s\neq i$ is a sink and $B_s=0$. 

\begin{lemma}\label{nnlem5015}
With the notation above, the space  $B'=B\oplus\Bbbk\{\mathtt{a}_{si}\}$ is a subbimodule of ${}_AA_A$.
\end{lemma}

\begin{proof}
By construction, $\Bbbk\{\mathtt{a}_{si}\}$ is a socle element of $P_i$. Therefore $B'$ is a left module.
As $i$ is a source, right multiplication with $\mathtt{a}_{ab}$ either preserves $\mathtt{a}_{si}$,
if $a=b=i$, or annihilates it in all other cases. The claim follows.
\end{proof}

Using \eqref{eq71} and $B_s=0$, we get $B=J_sB'$. As $B\subsetneq B'$, we
have $B'\in S$ by maximality of $B$. Therefore $B\in S$, a contradiction.
This completes the proof of Case~1.

{\bf Case~2.}
Now we may assume that there is an $i$ such that $B_i\not\in\{0,P_i\}$. In this case we may choose
$i\in\{1,2,\dots,n\}$ such that $B_i\not\in\{0,P_i\}$ and, additionally, for any $j$ for which 
there is an arrow $j\to i$ we have $B_j=P_j$ (note that such $B_j$ is automatically non-zero as $B_i\neq 0$
and $B$ is a subbimodule).  We again consider two subcases.

{\bf Subcase~2.A.}
Assume that we may choose such $i$ with the additional property that it is not a source. 
Then there is a unique $s\in\overline{i}$ such that $\mathrm{Rad}(P_s)=B_i$. 
Then $s$ is not a sink since $B_i\neq 0$. If $t\in \overline{i}$, then $B_t$ cannot have 
$P_s$ as a direct summand as $B_i\subset \mathrm{Rad}(P_s)$. If $t\not\in \overline{i}$, then 
$B_t$ cannot have $P_s$ as a direct summand as $B_j=P_j$ for all $j$ which have an arrow to $i$. 
Therefore, similarly to the above, $B'=B\oplus\Bbbk\{\mathtt{a}_{si}\}$ is 
a strictly larger  subbimodule than $B$ and $B=J_sB'$, leading again to a contradiction.

{\bf Subcase~2.B.}
Finally, consider the subcase when any $i$ as above in Case~2 is a source. In this case we have 
$J_iM=M$ for any left submodule $M\subset J_i$
by \eqref{eq71}. If $\mathrm{deg}_Q(i)=1$, then $P_i$ is uniserial and
there is a unique $s\in\overline{i}$ such that the radical of $P_s$ is isomorphic to $B_i$. Then
$s$ is not a sink (since $B_i\neq 0$) and hence,
similarly to the above,  $B':=B\oplus\Bbbk\{\mathtt{a}_{si}\}$ is a subbimodule of
${}_AA_A$ and $B=J_sB'$ by \eqref{eq71}, a contradiction. 

It is left to assume $\mathrm{deg}_Q(i)>1$.
Then $J_i$ decomposes according to \eqref{eq452}. For $q=1,2,\dots,m_i$, set
\begin{displaymath}
B^{(q)} := B\cap J_i^{(q)}
\end{displaymath}
and denote by $t_q\in\Gamma^{(q)}$ the unique element such that there is an arrow $\alpha_q:i\to t_q$.
If $B^{(q)}_{i}=P_{t_q}$ for all $q$, we have $B_i=\mathrm{Rad}(P_i)$. Then, similarly to the above,
$B':=B\oplus\Bbbk\{\mathtt{a}_{ii}\}$ is a subbimodule of
${}_AA_A$ and $B=J_iB'$ by \eqref{eq71}, a contradiction. 

If $B^{(q)}_{i}\neq P_{t_q}$ for some $q$,  we note that $P_{t_q}$ is uniserial since $Q$
is admissible. There is a unique $s$ such that $\mathrm{Rad}(P_s)=B^{(q)}_{i}$.
Then, similarly to the above,
$B':=B\oplus\Bbbk\{\mathtt{a}_{si}\}$ is a subbimodule of
${}_AA_A$. If $s$ is not a sink, we also have $B=J_sB'$ by \eqref{eq71}, a contradiction.

If $s$ is a sink, we have $B^{(q)}_{i}=0$ and thus also $B^{(q)}_{s}=0$
and even $B^{(q)}_{t}=0$ for all $t$ in the maximal chain $X$ with source $i$ and sink $s$. In this case, 
$B^{(q)}$ is, naturally, a proper subbimodule of the identity bimodule for the 
path algebra $A_q$ of $\Gamma^{(q)}$. Note that $\Gamma^{(q)}$ has less than $n$ vertices.
Using induction by $n$ and the above arguments, we get that there are $x,y\in \Gamma^{(q)}$
such that $(B^{(q)})':=B^{(q)}\oplus\Bbbk\{\mathtt{a}_{xy}\}$ is an 
$A_q$-$A_q$--bimodule and $J^{(q)}_x (B^{(q)})'=B^{(q)}$, where $J^{(q)}_x$ denotes the 
kernel of the bimodule epimorphism $A_q\tto L_{xx}$. If $y\not\in X$,
then  $B':=B\oplus\Bbbk\{\mathtt{a}_{xy}\}$ is a subbimodule of
${}_AA_A$ and $B=J_xB'$ by the above, a contradiction. If $y\in X$,
then $x=s$ and $B':=B\oplus\Bbbk\{\mathtt{a}_{si}\}$ is a subbimodule of
${}_AA_A$ and $B=J_sB'$ by the above, a contradiction. 
Claim~\eqref{prop41.1} follows.
\end{proof}

\begin{proof}[Proof of claim~\eqref{prop41.2}.]
The fact that  $\mathbf{B}$ generates $\mathcal{I}^{\mathrm{ind}}$ follows from claim~\eqref{prop41.1} and
Proposition~\ref{prop32} since $\mathbf{B}$ is exactly the set of indecomposable summands of elements in $\mathbf{Q}$.

To prove minimality of $\mathbf{B}$, assume that we can write some $J_s^{(q)}$ as a 
product of other elements in $\mathbf{B}$. By Corollary~\ref{cor33-1}, we thus get that 
$\mathrm{supp}(J_s^{(q)})$ is the intersection of the supports of the factors in this product.
For each factor $D$ of this product, the degree of $s$ in $\mathrm{supp}(D)$ is either $1$
or equal to $\deg_Q(s)$. If none of the  factors has $s$ with degree $1$, then all factors
have it with degree $\deg_Q(s)$ and thus the degree of $s$ in the support of the product must 
be $\deg_Q(s)$ as well, a contradiction. Therefore the product contains at least one factor
$D$ for which $\deg_{\mathrm{supp}(D)}(s)=1$. This means that $D=J_s^{(q')}$ for some $q'$.
As $\mathrm{supp}(D)\supset \mathrm{supp}(J_s^{(q)})$, it follows that $q=q'$, a contradiction.
Minimality of $\mathbf{B}$ follows.

The argument from the previous paragraph even shows that, if
we write $J_s^{(q)}$ as a product of arbitrary elements in $\mathcal{I}^{\mathrm{ind}}$, then
one of the factors must be contained in  $J_s^{(q)}$. If the factor is properly contained in
$J_s^{(q)}$, then the whole product is properly contained in $J_s^{(q)}$ as well.
Therefore one of the factors equals $J_s^{(q)}$, which implies uniqueness of $\mathbf{B}$. 
This completes the proof of claim~\eqref{prop41.2}.
\end{proof}

\subsection{Relations}\label{s8.2}

\begin{proposition}\label{prop51}
The ideals $J_i$, $i=1,2,\dots,n$, satisfy the following relations:
\begin{enumerate}[$($a$)$]
\item\label{prop51.1} $J_i^2=J_i$ for all $i$.
\item\label{prop51.2} $J_iJ_j=J_jJ_i$ if there is no arrow between $i$ and $j$.
\item\label{prop51.3} $J_{j}J_iJ_j=J_iJ_jJ_i=J_jJ_i$ if there is an arrow $i\to j$.
\end{enumerate}
\end{proposition}

\begin{proof}
To prove claim~\eqref{prop51.1}, note that $P_i$ is not a direct summand of $J_i$ since $A$
is hereditary and finite dimensional. Therefore $J_i^2=J_i$ for all $i$ follows directly from 
Lemma~\ref{nnlem502}.

To prove claim~\eqref{prop51.2}, note that $(J_iJ_j)_s=(J_jJ_i)_s$ for all $s\in Q_0\setminus\{i,j\}$.
To compare $(J_iJ_j)_s$ with $(J_jJ_i)_s$ for $s\in \{i,j\}$, we observe that, in case there are no 
arrows between $i$ and $j$, $P_i$ is not isomorphic to a summand of $(J_j)_j$ and 
$P_j$ is not isomorphic to a summand of $(J_i)_i$. Therefore, Lemma~\ref{nnlem502} implies
\begin{displaymath}
(J_iJ_j)_i=(J_jJ_i)_i=\mathrm{Rad}(P_i)\quad \text{ and }\quad 
(J_iJ_j)_j=(J_jJ_i)_j=\mathrm{Rad}(P_j).
\end{displaymath}
This proves claim~\eqref{prop51.2}.

Similarly, to prove claim~\eqref{prop51.3}, it is enough to compare 
$(J_jJ_iJ_j)_s$ with $(J_iJ_jJ_i)_s$ for $s\in \{i,j\}$. From Lemma~\ref{nnlem502}, it follows that
\begin{displaymath}
(J_jJ_iJ_j)_j=(J_iJ_jJ_i)_j=(J_jJ_i)_j=\mathrm{Rad}(P_j)
\end{displaymath}
since there are no arrows from $j$ to $i$ or from $j$ to $j$. 
For $s=i$, we can write $\mathrm{Rad}(P_i)=P_j\oplus N$,
where $N$ contains no direct summand isomorphic to $P_i$ or $P_j$.
Therefore Lemma~\ref{nnlem502} implies that
\begin{displaymath}
(J_jJ_iJ_j)_i=(J_iJ_jJ_i)_i=(J_jJ_i)_i=\mathrm{Rad}(P_j)\oplus N.
\end{displaymath}
This completes the proof.
\end{proof}

\begin{proposition}\label{prop52}
For $i,j\not\in\mathbf{K}'(Q)$, $s,t\in\mathbf{K}'(Q)$, $q\in\{1,2,\dots,m_s\}$ and $p\in\{1,2,\dots,m_t\}$, the
elements of $\mathbf{B}$ satisfy the following:
\begin{enumerate}[$($a$)$]
\item\label{prop52.1} Relations from Proposition~\ref{prop51}\eqref{prop51.1}-\eqref{prop51.3}
for $i,j\not\in \mathbf{K}'(Q)$.
\item\label{prop52.2} $(J_s^{(q)})^2=J_s^{(q)}$.
\item\label{prop52.3} $J_s^{(q)}J_s^{(q')}=0$ for any $q'\in \{1,2,\dots,m_s\}$, $q'\neq q$.
\item\label{prop52.4} $J_s^{(q)}J_t^{(p)}=J_t^{(p)}J_s^{(q)}$ if there is no arrow between $s$ and $t$.
\item\label{prop52.45} $J_s^{(q)}J_i=J_iJ_s^{(q)}$ if there is no arrow between $s$ and $i$.
\item\label{prop52.5} $J_{t}^{(p)}J_s^{(q)}J_t^{(p)}=J_s^{(q)}J_t^{(p)}J_s^{(q)}=
J_t^{(p)}J_s^{(q)}$  if there is an arrow $s\to t$.
\item\label{prop52.55} $J_{t}^{(p)}J_iJ_t^{(p)}=J_iJ_t^{(p)}J_i=
J_t^{(p)}J_i$  if there is an arrow $i\to t$.
\item\label{prop52.56} $J_{t}^{(p)}J_iJ_t^{(p)}=J_iJ_t^{(p)}J_i=
J_iJ_t^{(p)}$ if there is an arrow $t\to i$.
\item\label{prop52.6} $J_{t}^{(p)}J_s^{(q)}J_t^{(p')}=0$ for any $p'\in \{1,2,\dots,m_t\}$, $p'\neq p$,
if there is an arrow between $s$ and $t$.
\item\label{prop52.65} $J_{t}^{(p)}J_iJ_t^{(p')}=0$ for any $p'\in \{1,2,\dots,m_t\}$, $p'\neq p$,
if there is an arrow between $t$ and $i$.
\item\label{prop52.9} $J_s^{(q)}J_t^{(p)}=J_{t}^{(p)}$ in case $\mathrm{supp}(J_t^{(p)})\subset 
\mathrm{supp}(J_s^{(q)})$.
\item\label{prop52.95} $J_s^{(q)}J_t^{(p)}=0$ in case $\mathrm{supp}(J_t^{(p)})\cap  
\mathrm{supp}(J_s^{(q)})=\varnothing$.
\item\label{prop52.97} $J_s^{(q)}J_i=J_iJ_s^{(q)}=J_s^{(q)}$ if $i\not\in\mathrm{supp}(J_s^{(q)})$.
\end{enumerate}
\end{proposition}

\begin{proof}
Relations~\eqref{prop52.1} are clear. From Proposition~\ref{prop51}\eqref{prop51.1}, 
for $s\in\mathbf{K}'(Q)$, we have
\begin{displaymath}
\big(\bigoplus_{q=1}^{m_s}J_s^{(q)}\big)^2\cong 
\bigoplus_{q=1}^{m_s}(J_s^{(q)})^2\oplus 
\bigoplus_{q\neq q'}(J_s^{(q)}J_s^{(q')})
\cong\bigoplus_{q=1}^{m_s}J_s^{(q)}.
\end{displaymath}
Note that $(J_s^{(q)})^2$ is the only direct summand in the middle with support in $\Gamma^{(q)}\cup\{s\}$
as $(J_s^{(q)})^2\neq 0$ by Corollary~\ref{cor33-1}. Therefore we get  relations~\eqref{prop52.2} 
and \eqref{prop52.3}. 

By Proposition~\ref{prop51}\eqref{prop51.2}, 
for $s,t\in\mathbf{K}'(Q)$ in case there is no arrow between $s$ and $t$, we have
\begin{displaymath}
\big(\bigoplus_{q=1}^{m_s}J_s^{(q)}\big)\big(\bigoplus_{p=1}^{m_t}J_t^{(p)}\big)=
\big(\bigoplus_{p=1}^{m_t}J_t^{(p)}\big)\big(\bigoplus_{q=1}^{m_s}J_s^{(q)}\big).
\end{displaymath}
Opening brackets and matching summands with the same support on the left hand side and on the right hand side,
we get relations~\eqref{prop52.4}. Relations~\eqref{prop52.45} are obtained similarly from
\begin{displaymath}
\big(\bigoplus_{q=1}^{m_s}J_s^{(q)}\big)J_i=J_i\big(\bigoplus_{q=1}^{m_s}J_s^{(q)}\big)
\end{displaymath}
which is again given by Proposition~\ref{prop51}\eqref{prop51.2}.

Relations~\eqref{prop52.5}--\eqref{prop52.65} are obtained similarly from Proposition~\ref{prop51}\eqref{prop51.3}.

To prove relation~\eqref{prop52.9}, we compare $J_s^{(q)}J_t^{(p)}P_r$ with $J_t^{(p)}P_r$, for $r\in\{1,2,\dots,n\}$,
using \eqref{eq71} and a similar formula for $J_t^{(p)}P_r$, namely
\begin{equation}\label{eqnow1}
J_t^{(p)}P_r=
\begin{cases}
0, & \text{ if } r\not\in \Gamma^{(q)}\cup\{t\};\\
P_r, &\text{ if } r\in \Gamma^{(q)};\\
P_{t_p}, & \text{ if } r=t.
\end{cases}
\end{equation}
Here in the last line we identify $P_{t_p}$ with the corresponding submodule in $\mathrm{Rad}(P_t)$.
Note that the only $P_i$ which appear (up to isomorphism)
as direct summands of $J_t^{(p)}P_r$ are those for which we have  $i\in\mathrm{supp}(J_t^{(p)})\setminus\{t\}$. 
If $\mathrm{supp}(J_t^{(p)})\subset \mathrm{supp}(J_s^{(q)})$,
then $s\not\in \mathrm{supp}(J_t^{(p)})\setminus\{t\}$ and hence
$J_s^{(q)}P_i=P_i$ for such $i$ by \eqref{eqnow1}. This implies relation~\eqref{prop52.9}. 
Similarly one checks relations~\eqref{prop52.95} and \eqref{prop52.97}. This completes the proof.
\end{proof}

\subsection{The main results}\label{s8.4}

\begin{theorem}\label{thm54}
Any relation between elements in the generating set $\mathbf{Q}$ of the monoid $\mathcal{I}$ 
is a consequence of the relations given in Proposition~\ref{prop51}.
\end{theorem}

\begin{proof}
Let $S$ be the abstract monoid given by generators $\mathbf{Q}$ and relations from  Proposition~\ref{prop51}.
Note that $S$ is a Hecke-Kiselman monoid in the sense of \cite{GM2}.
Then we have the canonical surjection $\psi:S\tto \mathcal{I}$.

For any $I\in\mathcal{I}$ the additive functor $\mathrm{Su}_I$ acting on $A\text{-}\mathrm{proj}$ defines
an endomorphism of the split Grothendieck group $[A\text{-}\mathrm{proj}]_{\oplus}$. This gives a homomorphism
$\varphi$ from $\mathcal{I}$ to the monoid of all endomorphisms of $[A\text{-}\mathrm{proj}]_{\oplus}$. Let
$T$ denote the image of $\varphi$. Combined together we have surjective composition $\varphi\circ\psi$
as follows: $S\tto \mathcal{I}\tto T$.

Consider the standard basis $\{[P_i]\,:\,i=1,2,\dots,n\}$ in $[A\text{-}\mathrm{proj}]_{\oplus}$. From
isomorphisms in \eqref{eq71}, for $i,j=1,2,\dots,m$ we obtain that $T$ is spanned by 
$\varphi(J_i)[P_j]$, where
\begin{displaymath}
\varphi(J_i)[P_j]=
\begin{cases}
[P_j],& i\neq j;\\ 
\displaystyle\sum_{i\to s}[P_s], & i=j. 
\end{cases}
\end{displaymath}
This is exactly the image of the linear representation of $S$ considered in \cite[Theorem~4.5]{Fo} 
where it was proved that the corresponding representation map is injective, that is $S\cong T$. Consequently, 
because of the sandwich position of $\mathcal{I}$ between $S$ and $T$, we obtain $S\cong \mathcal{I}$ 
and the proof is  complete.
\end{proof}

We note that in \cite{Gr} (see also \cite{Fo}) one finds a recursive description of a normal form for
elements of the Hecke-Kiselman monoid $S$ from the above proof of Theorem~\ref{thm54}. The same thus 
holds also for the monoid $\mathcal{I}$. A certain description of a normal form for elements in the 
monoid $\mathcal{I}^{\mathrm{ind}}$ follows from the proof of the following theorem which also
describes relations in $\mathcal{I}^{\mathrm{ind}}$.

\begin{theorem}\label{thm55}
Any relation between elements in the generating set $\mathbf{B}$ of the monoid $\mathcal{I}^{\mathrm{ind}}$ 
is a consequence of the relations given in Proposition~\ref{prop52}.
\end{theorem}

\begin{proof}
Let $S$ be the abstract monoid given by generators $\mathbf{X}:=\mathbf{B}\cup\{0\}$ and relations 
from Proposition~\ref{prop52} together with the obvious relations defining $0$ as the zero element.
As usual, we denote by $\mathbf{X}^+$ the set of all non-empty words in the alphabet $\mathbf{X}$.
Let $\psi:S\to \mathcal{I}^{\mathrm{ind}}$ be the obvious canonical surjection.
Our aim is to prove injectivity of $\psi$. For this we will describe certain 
representatives in the fibers of the canonical map $\tau:\mathbf{X}^+\to S$.

For simplicity, we call all elements in $\mathbf{X}$ of the form $J_s^{(q)}$, for $s\in\mathbf{K}'(Q)$
and $q\in\{1,2,\dots,m_s\}$, the {\em split symbols}. For $w\in \mathbf{X}^+$, let 
$J_{s_1}^{(q_1)},J_{s_2}^{(q_2)},\dots,J_{s_k}^{(q_k)}$ be the list of all split symbols 
which appear in $w$. If $w$ has no split symbols, we set $\Omega=Q$. If $w$ contains $0$, we set
$\Omega=\varnothing$. Otherwise, set
\begin{displaymath}
\Omega:=\bigcap_{i=1}^k\mathrm{supp}(J_{s_i}^{(q_i)}).
\end{displaymath}
Note that $\Omega$ is not an invariant of a fiber of $\tau$ in general. 

If $\Omega=\varnothing$, then either $w$ contains $0$ or 
the fact that $Q$ is a tree implies existence of $i,j\in \{1,2,\dots,k\}$ such that
$\mathrm{supp}(J_{s_i}^{(q_i)})\cap \mathrm{supp}(J_{s_j}^{(q_j)})=\varnothing$. We claim that in the latter 
case $w=0$ in $S$. Without loss of generality we may assume that the indices $i$ and $j$ and the word $w$ 
(in its equivalence class) are chosen such that $w=xJ_{s_i}^{(q_i)}yJ_{s_j}^{(q_j)}z$ 
with $y$ shortest possible. Assume that $y\neq 0$ and that
there is a $J_r$ with $r\in \mathrm{supp}(J_{s_j}^{(q_j)})$ 
in $y$. Take the leftmost occurrence of such element in $y$.
Now we may use relations in Proposition~\ref{prop51}\eqref{prop51.2} to move it
past all $J_q$ with $q\not\in \mathrm{supp}(J_{s_j}^{(q_j)})$.  Note that, by the minimality of $y$ and 
relations in Proposition~\ref{prop52}\eqref{prop52.4}, \eqref{prop52.45} and \eqref{prop52.9},
there is no $J_{s_a}^{(q_a)}$ between $J_{s_i}^{(q_i)}$ and $J_r$, such that $r$ and $s_a$
are connected by an arrow. So $J_r$ commutes with any split symbol between $J_{s_i}^{(q_i)}$ and
$J_r$ and so we can move it past $J_{s_i}^{(q_i)}$ making $y$ shorter. Therefore $y$
cannot contain any $J_r$ with $r\in \mathrm{supp}(J_{s_j}^{(q_j)})$.

Similarly, $y$ does not contain any $J_r$ with $r\in \mathrm{supp}(J_{s_i}^{(q_i)})$. Analogously 
(using also Proposition~\ref{prop52}\eqref{prop52.3}) one shows that $y$ does not contain any
split symbol $J_r^{(f)}$ for $r\in \mathrm{supp}(J_{s_i}^{(q_i)})\cup \mathrm{supp}(J_{s_j}^{(q_j)})$. 
Using similar arguments, 
it follows that $y$ may contain only elements $J_r$ 
where $r$ belong to the unique (unoriented) path between $s_i$ and $s_j$. Moreover, to avoid application of
a similar argument, all vertices from this path must occur. But then one can use, if necessary, 
relations in Proposition~\ref{prop52}\eqref{prop52.97} to make $y$ shorter. Hence $y$ is empty and we may use
relations in Proposition~\ref{prop52}\eqref{prop52.95} to conclude that $w=0$.

The case when $\Omega$ has only one vertex is dealt with similarly using 
relations in Proposition~\ref{prop52}\eqref{prop52.3}, \eqref{prop52.6} and \eqref{prop52.65}
and also results in $w=0$.

If $\Omega$ has at least two vertices, note that, for any split symbol $J_{s_i}^{(q_i)}$ in $w$, we have
$\mathrm{deg}_{\Omega}(s_i)\leq \mathrm{deg}_{\mathrm{supp}(J_{s_i}^{(q_i)})}(s_i)=1$.
Then a similar commutation and deleting procedure as above combined with the relations in
Proposition~\ref{prop52}\eqref{prop52.9} shows that $w$ can be changed to an equivalent word $u$ with the property
that the only split symbols in $u$ are those $J_{s_i}^{(q_i)}$ for which $s_i\in\mathbf{K}'(Q)$ and
$\deg_{\Omega}(s_i)=1$, moreover, each of them occurs exactly once. Furthermore, one can use
relations in Proposition~\ref{prop52}\eqref{prop52.97} and \ref{prop52}\eqref{prop52.45} 
to ensure that $u$ contains only $J_t$ for $t\in\Omega$. 

Let $\Omega'$ be the full subgraph of $\Omega$ with vertex set $\Omega\setminus \mathbf{K}'(Q)$. 
If $\Omega'$ is empty, then the above implies that $u$ is a product of split symbols
$J_{s_i}^{(q_i)}$ for which $s_i\in\mathbf{K}'(Q)$ and $\deg_{\Omega}(s_i)=1$. 
If $\Omega$ has two vertices, they are necessarily connected and  
we can use relation Proposition~\ref{prop52}\eqref{prop52.5} to see that there are exactly two possibilities for 
$u$, namely $J_{s_1}^{(q_1)}J_{s_2}^{(q_2)}$ and $J_{s_2}^{(q_2)}J_{s_1}^{(q_1)}$. These two elements are
different in $\mathcal{I}^{\mathrm{ind}}$ since from \eqref{eq71} it follows that their actions on 
$P_{s_1}\oplus P_{s_2}$ are different. Using these actions, we also see that these two elements 
differ from $0$ in $S$.  If $\Omega$ has more than two vertices, then all factors in 
$u$ commute and thus define $u$ uniquely.

It remains to consider the case when $\Omega'$ is non-empty. 
Since $Q$ is admissible, $\Omega'$ is a disjoint union of graphs of the form \eqref{eq2}. Let
$\Gamma_1,\Gamma_2,\dots,\Gamma_m$ be the connected components of $\Omega'$. Using relations given by 
Proposition~\ref{prop52}\eqref{prop52.4} and \eqref{prop52.45},  we can write
$u=u_1u_2\dots u_m$, where each element $u_r$, for $r=1,2,\dots,m$, 
is a product of $J_i$ or $J_s^{(q)}$ with $i,s\in \Gamma_r$
such that $\mathrm{deg}_{\Gamma_r}(i)=2$ and  $\mathrm{deg}_{\Omega}(s)=1$.
We also have that all factors $u_r$ commute with each other. It remains to show that, if
$u'=u'_1u'_2\dots u'_m$ is another word with similar properties which defines the same element
in $\mathcal{I}^{\mathrm{ind}}$, then, up to permutation of factors, we have that 
$u_r$ is equivalent to $u'_r$ in $S$ for each $r$.

For a fixed $r$, we are in the situation of the quiver \eqref{eq2}.
The relations from Proposition~\ref{prop52}\eqref{prop52.1}, \eqref{prop52.2}, \eqref{prop52.45},
\eqref{prop52.55} and \eqref{prop52.56} guarantee that the $J_i$'s and $J_s^{(q)}$'s,
where $i,s\in \Gamma_r$, satisfy all relations for the corresponding Hecke-Kiselman monoid 
of type $A$ as defined in \cite{GM2}. Let $S_r$ be the submonoid of $S$ generated by 
$J_i$ and $J_s^{(q)}$, where $i,s\in \Gamma_r$. 
Now we can proceed similarly to the proof of Theorem~\ref{thm54}. 

Consider the standard basis $\{[P_i]\,:\,i=1,2,\dots,n\}$ in $[A\text{-}\mathrm{proj}]_{\oplus}$.
The action of $\mathcal{I}^{\mathrm{ind}}$ on $A\text{-}\mathrm{proj}$ induces a homomorphism 
from $\mathcal{I}^{\mathrm{ind}}$ to the monoid $\mathrm{End}([A\text{-}\mathrm{proj}]_{\oplus})$.
This induces a representation of $S_r$ in $\mathrm{End}([A\text{-}\mathrm{proj}]_{\oplus})$.
From \cite[Subsection~3.2]{GM2} it follows that the latter representation map is injective. 
This means that, in the situation above,  $u_r$ is equivalent to $u'_r$ in $S$
and hence the statement of the theorem follows.
\end{proof}

We do not know whether the systems of relations described in Theorems~\ref{thm54} and \ref{thm55} 
are minimal or not.

\vspace{1mm}

\noindent
A.-L.~G.: Faculty of Mathematics, Bielefeld University
PO Box 100 131, D-33501, Bielefeld, Germany, 
{\tt apaasch\symbol{64}math.uni-bielefeld.de}
\vspace{0.1cm}

\noindent
V.~M.: Department of Mathematics, Uppsala University,
Box 480, SE-75106, Uppsala, SWEDEN,  {\tt mazor\symbol{64}math.uu.se}; 
http://www.math.uu.se/{\small$\sim$}mazor/
\vspace{0.1cm}

\end{document}